\documentclass[12pt]{amsart}

\usepackage{tikz}
\usepackage{amsmath}
\usepackage{amsthm}
\usepackage{hyperref}
\usepackage[margin=1.2in]{geometry}

\numberwithin{equation}{section}

\newtheorem{prop}{Proposition}
\newtheorem{lemma}[prop]{Lemma}

\newtheorem{thm}[prop]{Theorem}

\newtheorem{conj}[prop]{Conjecture}
\numberwithin{prop}{section}

\theoremstyle{definition}
\newtheorem{defn}[prop]{Definition}

\newtheorem{rmk}[prop]{Remark}

\newcommand{\del}{\partial}
\newcommand{\delb}{\bar{\partial}}\newcommand{\dt}{\frac{\partial}{\partial t}}
\newcommand{\brs}[1]{\left| #1 \right|}

\renewcommand{\gg}{\gamma}
\newcommand{\gD}{\Delta}
\newcommand{\gd}{\delta}

\newcommand{\gk}{\kappa}

\newcommand{\gw}{\omega}
\newcommand{\ga}{\alpha}
\newcommand{\gb}{\beta}
\renewcommand{\ge}{\epsilon}
\newcommand{\N}{\nabla}
\newcommand{\FF}{\mathcal F}

\newcommand{\CC}{\mathcal C}
\newcommand{\DD}{\mathcal D}

\newcommand{\EE}{\mathcal E}
\newcommand{\LL}{\mathcal L}
\newcommand{\VV}{\mathcal V}
\renewcommand{\SS}{\mathcal S}
\newcommand{\RR}{\mathcal R}
\newcommand{\til}[1]{\widetilde{#1}}
\newcommand{\nm}[2]{\brs{\brs{ #1}}_{#2}}

\newcommand{\ohat}[1]{\overset{\circ}{#1}}

\renewcommand{\bar}[1]{\overline{#1}}

\newcommand{\HH}{\mathcal H}

\newcommand{\IP}[1]{\left<#1 \right>}

\newcommand{\floor}[1]{\lfloor #1 \rfloor}

\DeclareMathOperator{\PSH}{PSH}

\DeclareMathOperator{\tr}{tr}

\DeclareMathOperator{\Vol}{Vol}

\DeclareMathOperator{\supp}{supp}

\begin{document}

\title[Minimizing Movement solutions of Calabi flow]{Long time existence of
Minimizing Movement
solutions of Calabi flow}

\begin{abstract} We recast the Calabi flow in DeGiorgi's language
of
minimizing movements.  We establish the long time existence of minimizing
movements for K-energy with arbitrary initial condition.  Furthermore we
establish some a priori regularity results for these solutions, and that sufficiently
regular minimizing movements are smooth solutions to Calabi flow.
\end{abstract}

\author{Jeffrey Streets}
\address{Rowland Hall\\
         University of California, Irvine\\
         Irvine, CA 92617}
\email{\href{mailto:jstreets@uci.edu}{jstreets@uci.edu}}

\date{August 25, 2012}

\maketitle

\section{Introduction}

Let $(M^{2n}, \gw, J)$ be a compact K\"ahler manifold.  Fix $\phi \in
C^{\infty}(M)$ such that $\gw_{\phi} := \gw + \sqrt{-1} \del
\delb \phi > 0$, and let $s_{\phi}$ denote the scalar curvature of the metric
$\gw_{\phi}$.  Furthermore, let $V = \Vol(M)$, and set $\bar{s} = \frac{1}{V}
\int_M
s_{\phi} \gw_{\phi}^n$, which is fixed for any $\phi$.  A one-parameter family
of K\"ahler potentials $\phi_t$ is a solution of \emph{Calabi flow}
if
\begin{align} \label{Cflow}
\dt \phi =&\ s_{\phi} - \bar{s}.
\end{align}
\noindent This flow was introduced by Calabi in his seminal paper
\cite{CalabiExtremal} on extremal K\"ahler metrics.  Since then several
regularity and long time existence results have been obtained.  Long time
existence and convergence to a metric of constant scalar curvature on Riemann
surfaces was shown by Chrusciel \cite{Cru}.  A more direct proof using a
concentration/compactness argument was given by Chen \cite{ChenCFOS} (see also
\cite{Struwe}).  On complex surfaces long time existence and convergence
results have been obtained for certain metrics with small energy and toric
symmetry \cite{ChenHe}.  More recent work by Huang \cite{Huang1}, \cite{Huang2}
approaches the general problem of Calabi flow on toric varieties by exploiting
techniques used by Donaldson in understanding the constant scalar curvature
equation on such manifolds via a continuity method.

The main motivating conjecture regarding the long time behavior of Calabi flow
is simple but ambitious.
\begin{conj} \label{existconj} (Calabi, Chen) Let $(M^{2n}, \gw, J)$ be a
compact K\"ahler
manifold.  The
solution to the Calabi flow with any initial condition exists smoothly on
$[0, \infty)$.
\end{conj}
\noindent Furthermore, there are conjectures on the nature of the singularity
formation at
infinity.  One example is the following.
\begin{conj} \label{convconj} (Donaldson) Let $(M^{2n}, \gw, J)$ be a compact
K\"ahler manifold
of constant scalar curvature.  The solution to the
Calabi flow with any initial condition exists for all time and converges to
a constant scalar curvature metric.
\end{conj}

The main purpose of this paper is to prove the long time existence of a certain
kind of weak solution to the
Calabi flow known as a ``minimizing movement" in the terminology of DeGiorgi
\cite{DeG1}.  These are solutions
constructed as limits of time-discretized  flows generated by an implicit Euler
method.  This technique involves freezing the time parameter of the gradient
flow
and constructing small time-step approximations of the flow as critical points
of
certain distance-penalizing modifications of the functional in question.  This
is a very general framework for constructing gradient flows of
functionals in metric spaces which has been significantly expanded recently in
\cite{Ambrosio}.  We will exploit an earlier instance of this methodology,
namely a 
general existence result of Mayer \cite{Mayer}, extending the Crandall-Liggett
generation theorem \cite{CL} to the setting of metric spaces with
nonpositive curvature.

Recall that Calabi flow, while conceived as the gradient flow
of the Calabi energy, is also the gradient flow
of Mabuchi's $K$-energy functional.  As it turns out, the
$K$-energy on the space of K\"ahler metrics, denoted $\HH$, has many of the
formal properties which are
usually needed in this general
setup.  For instance, the $K$-energy is convex along smooth geodesics, and
the
space of K\"ahler metrics, endowed with the Mabuchi metric, has 
nonpositive curvature, making
the distance function also convex.  This makes it quite natural to approach
existence questions
for the Calabi flow using minimizing movements.  However, due to the lack of
regularity of geodesics and the incompleteness of $\HH$, care is required in
setting things up properly.  In preparing our application of
Mayer's theorem a crucial role is played
throughout by the theory of geodesics in $\HH$ developed
in the work of Chen, \cite{ChenSOKM}, \cite{ChenSOKM3},
Calabi-Chen \cite{CalabiChen}, and Chen-Tian \cite{ChenTian}.  

\begin{thm} \label{fundexistthm} Let $(M^{2n}, \gw, J)$ be a compact K\"ahler
manifold.  Given $\phi_0 \in \bar{\HH}$ there exists a $K$-energy minimizing
movement
$\phi : [0,\infty) \to \bar{\HH}$ with initial condition $\phi_0$.
\end{thm}

The definition of a $K$-energy minimizing movement appears in \S \ref{WKSLNS}. 
These solutions come with a host of properties exhibiting the manner in which
they can be thought of as gradient lines for $K$-energy, and these are shown in
\S \ref{furtherprops}.  Note in particular that Theorem \ref{fundexistthm}
allows for the definition of a ``flow map'' $F : [0,\infty) \times \bar{\HH} \to
\bar{\HH}$.  This flow map satisfies the semigroup property, and is H\"older
continuous of exponent $\frac{1}{2}$ in the time variable (Theorem
\ref{Fprops}).  Furthermore, we show that for all $t$, $F(t, \cdot)
: \bar{\HH} \to \bar{\HH}$ is distance nonincreasing.  This property was shown
for smooth solutions to Calabi flow by Calabi-Chen (\cite{CalabiChen} Theorem
1.3, Theorem 1.5).

\begin{thm} \label{flowcontr1} Let $(M^{2n}, \gw,
J)$ be a compact K\"ahler manifold.  Given $\phi_0, \psi_0 \in \bar{\HH}$, for
all $t \geq 0$ one has
\begin{align*}
 d(\phi_t, \psi_t) \leq d(\phi_0, \psi_0).
\end{align*}
\end{thm}

As detailed in \S \ref{ltesec}, this solution guaranteed by Theorem
\ref{fundexistthm} is a
path
through $\bar{\HH}$, the completion of $\HH$ with respect to the distance
topology.  This alone guarantees little regularity for $\phi$ beyond what comes
automatically from the regularity of closed positive (1,1) currents.  Ideally
one would like to show that this minimizing movement solution
is in fact smooth, and moreover a solution to Calabi flow.  We take two steps in
this
direction, again exploiting the theory of geodesics in $\HH$.  The
first is to establish some extra regularity for the minimizing movements in the
case $c_1 < 0$ beyond
what is guaranteed by Mayer's theorem.

\begin{thm} \label{highreg} Let $(M^{2n}, \gw, J)$ be a compact K\"ahler
manifold satisfying $c_1 < 0$.  Given $\phi_0 \in \bar{\HH}$, the K-energy
minimizing movement with
initial
condition $\phi_0$ is a map
\begin{align*}
 \phi : [0,\infty) \to \bar{\HH} \cap H_1^2 \cap \PSH(M, J).
\end{align*}
\end{thm}

Lastly we establish the fact that, provided the discretized Calabi flows used to
generate the minimizing movement satisfy
sufficient a priori estimates, the limiting path through $\bar{\HH}$ is in fact
a
smooth solution to Calabi flow.  See Theorem \ref{convthm} for the precise
statement.

\begin{thm} \label{fundconv} A sequence of discrete Calabi flows with step size
approaching zero and satisfying sufficient a priori estimates contains a
subsequence converging to a smooth solution of Calabi flow.
\end{thm}

\begin{rmk} A different approximation scheme for Calabi flow on projective
varieties was considered by Fine \cite{Fine}. 
Fine's construction defines an ODE on maps defining projective embeddings of
the underlying complex manifold, called balancing flow.  This is a
generalization to the Calabi flow of the techniques used by Donaldson in
constructing
cscK metrics.  The main result shows that this sequence of ODE's, with
appropriate initial conditions, converges to a solution to Calabi flow, as long
as that solution exists smoothly.  To roughly compare this approach to ours,
Fine uses a natural sequence of finite dimensional approximations of the space
of K\"ahler metrics to approximate the flow by ODEs,
whereas we deal directly with this infinite dimensional space, but discretize
the time variable of the flow.
\end{rmk}

Here is an outline of the rest of the paper.  In \S \ref{SOKM} we review some
fundamental facts on the Mabuchi-Semmes-Donaldson metric, and we continue in \S
\ref{GEOD} with a thorough discussion of the structure of geodesics in this
metric, including Chen's $\ge$-geodesics and Chen-Tian's almost smooth
geodesics.  Then in \S \ref{WKSLNS} we define the Moreau-Yosida approximations
of $K$-energy and set up
the notion of a discrete Calabi flow and a minimizing movement for $K$-energy. 
Then in \S \ref{ltesec} we recall Mayer's theorem and prove Theorem
\ref{fundexistthm}.  Section \ref{highregsec} has the proof of Theorem
\ref{highreg}, and Theorem \ref{fundconv} is proved in \S \ref{fundconvsec}.

\textbf{Acknowledgements} The author would like to thank Will Cavendish and
Weiyong He for
several helpful conversations.  The author would especially like to thank
Patrick Guidotti for a number of helpful discussions on
gradient flows in metric spaces and for directing the author to the work of Uwe
Mayer \cite{Mayer}.  Lastly the author would like to thank the referee for a
careful reading.

\section{The space of K\"ahler metrics} \label{SOKM}

In this section we recall some fundamental properties of the
Mabuchi-Semmes-Donaldson metric (\cite{MabuchiSymp}, \cite{Semmes},
\cite{Donaldson}) on a K\"ahler class, and some important functionals on this
space.  First we recall the
definition of the metric and the interpretation of a K\"ahler class as an
infinite dimensional symmetric space with nonpositive curvature.

\begin{defn} Let $(M^{2n}, \gw, J)$ be a K\"ahler manifold.  Let
\begin{align*}
\HH = \{\phi \in C^{\infty}(M) \ | \  \gw_{\phi} = \gw + \sqrt{-1} \del \delb
\phi > 0 \}.
\end{align*}
This space is often denoted $\HH_{[\gw]}$ to indicate the dependence on the
underlying K\"ahler class, but we have omitted this for notational simplicity,
and consider a given compact K\"ahler manifold with given K\"ahler class as
fixed throughout.
\end{defn}

\noindent Given $\phi \in \HH$, one has $T \HH_{\phi} \cong C^{\infty}(M)$.  
One can use the $L^2$ inner product associated to $\phi$ to define an inner
product on
$T\HH_{\phi}$.  In particular, given $f_1, f_2 \in T \HH_{\phi}$, let
\begin{align*}
\left< f_1, f_2 \right>_{\phi} :=&\ \int_M f_1 f_2 \gw_{\phi}^n.
\end{align*}
This defines a Riemannian metric on $\HH$.  By calculating formally (see \S
\ref{GEOD}) one obtains the
geodesic
equation of a path $\phi_t \in \HH$:
\begin{gather} \label{geod}
\ddot{\phi} = \frac{1}{2} \brs{ \N \dot{\phi}}^2_{\phi(t)},
\end{gather}
where $\dot{\phi} = \frac{\del \phi}{\del t}$, etc.

The Levi-Civita connection of this metric is most easily defined in terms of
differentiation of vector
fields along paths.  In particular, if $\phi(t)$ is a one-parameter family of
$\HH$ and $\psi(t)$ is a smooth family of tangent vectors along
this path, we set
\begin{align} \label{dirdef}
\frac{D}{\del t} \psi = \dot{\psi} - \frac{1}{2} \left< \N \psi, \N \dot{\phi}
\right>_{\phi}.
\end{align}
This connection satisfies the compatibility condition
\begin{align*}
\frac{\del}{\del t} \left< \psi_1, \psi_2 \right>_{\phi} =&\ \left<
\frac{D}{\del t} \psi_1, \psi_2 \right>_{\phi} + \left< \psi_1, \frac{D}{\del t}
\psi_2 \right>_{\phi}.
\end{align*}

\noindent Amazingly, this metric gives $\HH$ the structure of an infinite
dimensional symmetric space of nonpositive curvature.  This fact, through its
manifestation in Lemma \ref{triangleineq}, plays an essential role in the proof
of Theorem \ref{fundexistthm}.  We record the specific curvature calculation for
completeness.

\begin{lemma} \label{curvaturecalc} (\cite{MabuchiSymp} Theorem 4.3) Suppose
$\ga, \gb \in T_{\phi} \HH$.  Then
\begin{align*}
R(\ga, \gb) \gg =&\ -  \{ \{\ga, \gb \}_{\phi}, \gg
\}_{\phi}
\end{align*}
where $\{ \cdot, \cdot \}_{\phi}$ denotes the Poisson bracket associated to the
symplectic manifold $(M, \gw_{\phi})$.  In particular, the sectional curvatures
satisfy
\begin{align*}
\left< R(\ga, \gb) \gb, \ga \right> =&\ - \nm{ \{\ga, \gb\}_{\phi}}{L^2}^2.
\end{align*}
\end{lemma}

\subsection{Decomposition of \texorpdfstring{$\HH$}{H}}

There is a natural decomposition
\begin{align*}
T \mathcal H_{\phi} = \left\{ \psi \left| \int_M \psi dV_{\phi} = 0
\right.\right\} \oplus \mathbb R.
\end{align*}
We can naturally decompose the entire space $\mathcal H = \mathcal H_0 \times
\mathbb R$ according to this decomposition.  In particular, let $\ga$ be the
$1$-form on $\mathcal H$ defined by
\begin{align*}
\ga_{\phi}(\psi) = \int_M \psi dV_{\phi}.
\end{align*}
One can check that $\ga$ is closed, and moreover that there exists a unique
function $I : \mathcal H \to \mathbb R$ such that $I(0) = 0$ and $\ga = dI$.  By
directly integrating one computes
\begin{align} \label{Iformula2}
I(\phi) = \sum_{j = 0}^n \frac{1}{(j+1)!(n-j)!} \int_M \phi \gw^{n-j} \wedge
\left(\sqrt{-1} \del
\delb \phi \right)^j.
\end{align}
This functional $I$ has an interpretation as a an integral along paths akin to
the $K$-energy.  Specifically, it follows from a calculation in (\cite{ChenLBM}
pg. 615) that in fact
\begin{align} \label{Iformula}
I(\phi) = \frac{1}{V} \int_0^1 \int_M \dot{\phi} \gw_{\phi}^n dt
\end{align}
where $\phi : [0,1] \to \HH$ is a smooth one-parameter family satisfying $\phi_0
= 0$, $\phi_1 =
\phi$.

\begin{rmk} It follows from (\ref{Iformula}) that the Calabi flow preserves
$\HH_0$, and thus it is natural to restrict our attention entirely to this
space.
\end{rmk}

\subsection{Functionals on \texorpdfstring{$\HH$}{H}}

In this subsection we recall some functionals on $\HH$ and some of their
properties for convenience.

\begin{defn} Let $(M^{2n}, \gw, J)$ be a compact K\"ahler manifold.  The
\emph{$K$-energy functional} is
\begin{align} \label{Kenergy}
\nu(\phi) =&\ - \int_0^1 \int_M (s(\gw_{\phi}) - \bar{s}) \dot{\phi}
\gw_{\phi}^{\wedge n} dt,
\end{align}
where $\phi :[0,1] \to \HH$ is a smooth one-parameter family satisfying $\phi_0
= 0$, $\phi_1 = \phi$.
\end{defn}

\begin{rmk} Herein we do not concern ourselves with the ``modified'' $K$-energy
which is used in the presence of a continuous group of automorphisms of $(M, J)$
(see \cite{Calabi}, \cite{Futaki}).  The results of this paper apply to
constructing minimizing movements for this modified $K$-energy as well, although
we do not explicitly do this here.  We have a very brief discussion of the
convergence properties of minimizing movements below, and it is likely that use
of the modified $K$-energy will be necessary in deriving more precise
convergence results.
\end{rmk}

\begin{defn} Let $(M^{2n}, \gw, J)$ be a compact K\"ahler manifold.  Let
\begin{align} \label{Jformula}
J(\phi) =&\ - \frac{1}{(n-1)!} \int_0^1 \int_M \dot{\phi} \rho(\gw) \wedge
\gw_{\phi_t}^{n-1},
\end{align}
where $\phi : [0,1] \to \HH$ is a smooth one-parameter family satisfying $\phi_0
= 0, \phi_1 =
\phi$.  By integrating along a straight line path one obtains the explicit
formula
\begin{align} \label{Jformula2}
 J(\phi) =&\ - \sum_{j=0}^{n-1} \frac{1}{(j+1)!(n-j-1)!} \int_M \phi \rho(\gw)
\wedge \gw^{n-j-1} \wedge \left( \sqrt{-1} \del \delb \phi \right)^j.
\end{align}
\end{defn}

\begin{defn} Let $(M^{2n}, \gw, J)$ be a compact K\"ahler manifold.  Let
\begin{align} \label{IAdef}
I^A(\phi) =&\ \frac{1}{V} \int_M \phi \left( \gw^n - \gw_{\phi}^n \right) =
\frac{1}{V} \sum_{i=0}^{n-1} \int_M \sqrt{-1} \del \phi \wedge \delb \phi \wedge
\gw^i \wedge \gw_{\phi}^{n-1-i},\\ \label{JAdef}
J^A(\phi) =&\ \frac{1}{V} \sum_{i=0}^{n-1} \frac{i+1}{n+1} \int_M \sqrt{-1} \del
\phi \wedge \delb \phi \wedge \gw^i \wedge \gw_{\phi}^{n-1-i}.
\end{align}
\end{defn}

\begin{rmk} We have decorated these functionals with the superscript $A$, since
these functionals typically are denoted by $I$ and $J$, but so unfortunately are
the functionals defined in (\ref{Iformula}), (\ref{Jformula}).
\end{rmk}

\noindent These functionals satisfy the inequality
\begin{align*}
\frac{1}{n} J^A(\phi) \leq I^A(\phi) - J^A(\phi) \leq n J^A(\phi).
\end{align*}

The next lemma gives an explicit form for the $K$-energy, which is arrived at by
evaluating the definition (\ref{Kenergy}) along linear paths.  This formulation
has the further benefit of showing that the $K$-energy is well-defined for
$C^{1,1}$ limits of metrics.
\begin{lemma} (\cite{ChenLBM} pg. 1) Let $(M^{2n}, \gw, J)$ be a compact
K\"ahler manifold.  Then
\begin{align} \label{Kenergform}
\nu(\phi) = \int_M \log \frac{\gw_{\phi}^n}{\gw^n} \gw_{\phi}^n + J(\phi)
+ \bar{s}
I(\phi).
\end{align}
\end{lemma}

\noindent Next we record a crucial observation regarding the $K$-energy, namely
its
convexity along smooth geodesics in $\HH$.

\begin{lemma} \label{Kenergyconv} (\cite{Mabuchi} Theorem 6.2) Let $\phi_t$ be a
path in $\mathcal H$.  Then
\begin{align*}
\frac{d^2 \nu(\phi)}{d t^2} =&\ \frac{1}{V} \nm{\delb \N^{1,0}
\dot{\phi}}{L^2(\gw_{\phi})}^2 - \int_M \left( \ddot{\phi} - \frac{1}{2} \brs{\N
\dot{\phi}}^2 \right) (s_{\phi} - \bar{s}) dV_{\phi}.
\end{align*}
In particular, the $K$-energy is weakly convex along a smooth geodesic.
\end{lemma}

\section{The structure of geodesics} \label{GEOD}
\subsection{Approximate geodesics and properties of the distance function}
\label{approxgeodss}

In this section we recall various results on geodesics in the space of K\"ahler
metrics.  We begin with the basic definitions and some variational formulas. 
Then we recall Chen's theory of $\ge$-approximate geodesics \cite{ChenSOKM},
which suffice for proving
most statements concerning the distance function on $\HH$.  In the next
subsection we will recall some fundamental facts about the more profound
regularity results concerning the ``partially smooth'' and ``almost smooth''
geodesics of \cite{ChenTian}.  The starting point is to define the
infinitesimal energy element for a path in $\HH$.  The fundamental definitions
are formally identical to those from Riemannian geometry.

\begin{defn} Let $\gg : [0,1] \to \HH$ be a smooth path.  The \emph{energy
element along $\gg$} is
\begin{align*}
E(t) := \int_M \brs{\frac{\del \gg}{\del t}(t)}^2 \gw_{\gg(t)}^n.
\end{align*}
\end{defn}

\begin{defn} Given $\gg : [0,1] \to \HH$ a smooth path, the \emph{length} of
$\gg$ is
\begin{align*}
\mathcal L(\gg) := \int_0^1 E(t)^{\frac{1}{2}} dt.
\end{align*}
Moreover, the \emph{energy} of a path in $\HH$ is
\begin{align*}
\EE(\gg) := \int_0^1 E(t) dt.
\end{align*}
\end{defn}

\begin{defn} Given $\phi,\psi \in \HH$, the \emph{distance from $\phi$ to
$\psi$} is
 \begin{align*}
  d(\phi,\psi) =&\ \inf_{\gg : \phi \to \psi} \LL(\gg).
 \end{align*}
\end{defn}

We next compute the first variation of the length integral.  This was done in
(\cite{Mabuchi} Theorem
7.3) for endpoint-fixed variations, which will not suffice for our purposes. 
While the calculation is formally identical to the usual calculation in
Riemannian geometry, we include it
for convenience.

\begin{lemma} \label{lengthvariations} Let $\gg_i : [0,1] \to \HH, i = 1,2$ be
smooth
paths, and suppose
\begin{align*}
\gg(t,s) : [0,1] \times [0,1] \to \HH
\end{align*}
is a smooth family of curves connecting $\gg_1(s)$ to $\gg_2(s)$.  Then
\begin{align*}
\frac{d \LL(\gg(\cdot, s))}{ds}  =&\ \left. E(s,t)^{-\frac{1}{2}}
\IP{\IP{\frac{\del \gg}{\del s}, \frac{\del \gg}{\del t}}}_{\gg(t,s)}
\right|_{t=0}^{t=1}\\
&\ \qquad + \int_0^1 E(s,t)^{-\frac{1}{2}} \IP{\IP{\frac{D}{\del t} \frac{\del
\gg}{\del t}, \frac{\IP{\IP{\frac{\del \gg}{\del s}, \frac{\del \gg}{\del
t}}}}{\IP{\IP{\frac{\del \gg}{\del t}, \frac{\del \gg}{\del t}}}} \frac{\del
\gg}{\del t} - \frac{\del \gg}{\del s}}} dt.
\end{align*}
\begin{proof} We directly compute
\begin{align*}
\frac{d \LL(\gg(\cdot, s))}{ds} =&\ \frac{d}{ds} \int_0^1 \left[
\IP{\IP{\frac{\del \gg}{\del t}, \frac{\del \gg}{\del t}}} \right]^{\frac{1}{2}}
dt\\
=&\ \int_0^1 E(s,t)^{-\frac{1}{2}} \IP{\IP{\frac{D}{\del s} \frac{\del \gg}{\del
t}, \frac{\del \gg}{\del t}}} dt\\
=&\ \int_0^1 E(s,t)^{-\frac{1}{2}} \IP{\IP{\frac{D}{\del t} \frac{\del \gg}{\del
s}, \frac{\del \gg}{\del t}}} dt\\
=&\ \int_0^1 E(s,t)^{-\frac{1}{2}} \left[ \frac{d}{dt} \IP{\IP{\frac{\del
\gg}{\del s}, \frac{\del \gg}{\del t}}} - \IP{\IP{\frac{\del \gg}{\del s},
\frac{D}{\del t} \frac{\del \gg}{\del t}}} \right] dt\\
=&\ \int_0^1 \left[ \frac{d}{dt} \left( E(s,t)^{-\frac{1}{2}} \IP{\IP{\frac{\del
\gg}{\del s}, \frac{\del \gg}{\del t}}} + E(s,t)^{-\frac{3}{2}}
\IP{\IP{\frac{\del \gg}{\del s}, \frac{\del \gg}{\del t}}} \IP{\IP{\frac{D}{\del
t}\frac{\del \gg}{\del t}, \frac{\del \gg}{\del t}}} \right)\right.\\
&\ \left. \qquad  - E(s,t)^{-\frac{1}{2}} \IP{\IP{\frac{\del \gg}{\del s},
\frac{D}{\del t} \frac{\del \gg}{\del t}}} \right]dt\\
=&\ \left. E(s,t)^{-\frac{1}{2}} \IP{\IP{\frac{\del \gg}{\del s}, \frac{\del
\gg}{\del t}}} \right|_{t=0}^{t=1} + \int_0^1 E(s,t)^{-\frac{1}{2}}
\IP{\IP{\frac{D}{\del t} \frac{\del \gg}{\del t}, \frac{\IP{\IP{\frac{\del
\gg}{\del s}, \frac{\del \gg}{\del t}}}}{\IP{\IP{\frac{\del \gg}{\del t}
\frac{\del \gg}{\del t}}}} \frac{\del \gg}{\del t} - \frac{\del \gg}{\del s}}}
dt.
\end{align*}
\end{proof}
\end{lemma}

This lemma in particular establishes the geodesic equation (\ref{geod}). 
However, as mentioned above, the lack of
full $C^{\infty}$ regularity results for geodesic paths in
$\mathcal H$ is an essential difficulty in working with $\HH$.
Nonetheless, by exploiting $\ge$-approximate geodesics, which we will define
below, one can establish foundational results on the metric space structure of
$\HH$.

\begin{defn} Let $(M^{2n}, \gw, J)$ a compact K\"ahler manifold, and fix
$\Omega$ a background volume form on $M$.  A one-parameter family
$\gg :[0,1] \to \HH$ is an \emph{$\ge$-approximate geodesic} if
\begin{align} \label{approxgeod}
\left( \ddot{\gg} - \brs{\N \dot{\gg}}^2_{\gg} \right) \gw_{\gg}^n = \ge
\Omega.
\end{align}
\end{defn}

\begin{lemma} \label{approxgeodlemma} (\cite{ChenSOKM} Lemma 7) Suppose $\gg_i :
[0,1] \to \HH, i = 1,2$ are two smooth curves.  There exists
$\ge_0$ small and a smooth two-parameter family of curves
\begin{align*}
\gg(t,s,\ge) : [0,1] \times [0,1] \times(0,\ge_0] \to \HH
\end{align*}
such that
\begin{enumerate} 
\item{For fixed $s_0$, the curve $\gg(t,s_0,\ge)$ is an $\ge$-approximate
geodesic from $\gg_1(s)$ to $\gg_2(s)$.}
\item{There exists $C$ depending on $\{\gg_i\}$ such that
\begin{align*}
\brs{\gg} + \brs{\frac{\del \gg}{\del s}} + \brs{\frac{\del \gg}{\del t}} < C,
\qquad 0 \leq \frac{\del^2 \gg}{\del t^2} < C, \qquad \frac{\del^2 \gg}{\del
s^2} < C.
\end{align*}}
\item{For fixed $s_0$, the curves $\{\gg(t,s_0, \ge)\}$ converge as $\ge \to 0$
to the unique geodesic connecting $\gg_1(s)$ to $\gg_2(s)$ in the weak $C^{1,1}$
topology.}
\item{There exists a constant $C$ depending on $\{\gg_i\}$ such that
\begin{align*}
\brs{\frac{\del E}{\del t}} \leq \ge C.
\end{align*}
That is, the energy element converges to a constant along each curve $\gg(t,s_0,
\ge)$ as $\ge \to 0$.}
\end{enumerate}
\end{lemma}

This lemma in particular yields the existence of a $C^{1,1}$ geodesic connecting
any two points.  A final point required in using these geodesics to prove that
$(\HH, d)$ is a metric space is to bound their length from below, and we record
this estimate as we will
use it several times below.

\begin{lemma} \label{distanceestimate}  Given $\psi, \phi \in \HH$ with
$I(\psi) = I(\phi)
= 0$, then
\begin{align*}
d(\psi, \phi) \geq V^{-\frac{1}{2}} \max \left\{ \int_{\phi - \psi > 0} (\phi -
\psi ) \gw_{\phi}^n, - \int_{\phi - \psi < 0} (\phi - \psi) \gw_{\psi}^n
\right\}.
\end{align*}
\begin{proof} By (\cite{ChenSOKM} Corollary 3) it suffices to estimate the
length of the $C^{1,1}$ geodesic connecting $\phi$ to $\psi$.  This proof
follows (\cite{ChenSOKM} Proposition 2) but with a more
general base point.  Let $\til{\gg}(t) = t \phi + (1-t) \psi$, and let $a(t) =
I(\til{\gg}(t))$.  It follows from (\ref{Iformula}) that
\begin{align*}
\dot{a}(t) =&\ \int_M \left(\phi - \psi \right) \gw^n_{\til{\gg}(t)},\\
\ddot{a}(t) =&\ \int_M \left( \phi - \psi \right) \gD_{\til{\gg}(t)} \left(\phi
-
\psi \right) \gw_{\til{\gg}(t)}^n \leq 0.
\end{align*}
Hence $\dot{a}(0) \geq a(1) - a(0) \geq \dot{a}(1)$, thus
\begin{align} \label{deloc10}
\int_M \left( \phi - \psi \right) \gw^n_{\psi} \geq I(\phi) - I(\psi) \geq
\int_M \left( \phi - \psi \right) \gw^n_{\phi}
\end{align}
Since $I(\phi) = I(\psi) = 0$ we conclude that $\phi - \psi$ attains positive
and negative values.  Now let $\gg_t$ be an $\ge$-approximate geodesic
connecting $\psi$ to $\phi$.  Note that one has $\ddot{\gg}(t) > \frac{1}{2}
\brs{\N
\dot{\gg}}^2_{\gg} \geq 0$.  It follows that
\begin{align} \label{deloc20}
\dot{\gg}(0) \leq \phi - \psi \leq \dot{\gg}(1).
\end{align}
Let $E(t)$ denote the infinitesimal energy along $\gg(t)$.  It follows from
H\"older's inequality and (\ref{deloc20}) that
\begin{align*}
\sqrt{E(1)} \geq&\ V^{-\frac{1}{2}} \int_M \brs{\dot{\gg}(1)}\\
\geq&\ V^{-\frac{1}{2}} \int_{\dot{\gg}(1) > 0} \dot{\gg}(1) \gw_{\phi}^n\\
\geq&\ V^{-\frac{1}{2}} \int_{\phi - \psi > 0} (\phi - \psi) \gw_{\phi}^n.
\end{align*}
Similarly one can derive
\begin{align*}
\sqrt{E(0)} \geq - V^{-\frac{1}{2}} \int_{\phi - \psi < 0} (\phi - \psi)
\gw_{\psi}^n.
\end{align*}
Note that by Lemma \ref{approxgeodlemma} part (4) it follows that for all $t_1,
t_2 \in
[0,1]$ one has $\brs{E(t_1) - E(t_2)} \leq C \ge$.  Hence for all $t$,
\begin{align*}
\sqrt{E(t)} \geq V^{-\frac{1}{2}} \max \left\{ \int_{\phi - \psi > 0} (\phi -
\psi ) \gw_{\phi}^n, - \int_{\phi - \psi < 0} (\phi - \psi) \gw_{\psi}^n
\right\} - C \ge.
\end{align*}
Integrating this inequality from $0$ to $1$ and sending $\ge$ to zero yields the
result.
\end{proof}
\end{lemma}

\begin{thm} (\cite{ChenSOKM} Corollary 3, Theorem 6)
\begin{enumerate}
\item{The space of K\"ahler potentials $\HH$ is convex by $C^{1,1}$ geodesics.}
\item{ $(\HH, d)$ is a metric space.}
\end{enumerate}
\end{thm}

Given that the geodesics  between points in $\HH$ pass through the closure of
$\HH$ with respect to weak $C^{1,1}$ convergence, it is convenient to work
directly with this space.  As we will also need to work with the metric space
completion of $\HH$, we take the time here to define notation for these two
spaces.  Note that in most of the literature $\bar{\HH}$ denotes the $C^{1,1}$
closure, whereas for us this denotes the metric space completion.
\begin{defn} \label{completiondef} Given $(M^{2n}, \gw, J)$ a compact K\"ahler
manifold, let
$\HH^{1,1}$ denote the closure of $\HH$ with respect to weak $C^{1,1}$
convergence.  Furthermore, let $(\bar{\HH}, \bar{d})$ denote the metric space
completion of $(\HH, d)$.  Note that weak $C^{1,1}$ convergence implies
strong $C^{1,\ga}$ convergence, which in turn by a Lemma \ref{toplemma} implies
convergence in the distance topology.  This means that one can interpret a point
in $\phi \in \HH^{1,1}$ as a point in $\bar{\HH}$ by mapping it to the
equivalence class of Cauchy sequences represented by any sequence converging to
$\phi$ in the weak $C^{1,1}$ topology.  We will do this implicitly in various
points in the paper.
\end{defn}

\begin{lemma} \label{toplemma} Let $\{\phi_n \} \in \HH$ be a sequence
converging in the weak $C^{1,1}$ topology.  Then $\{\phi_n\}$ is a Cauchy
sequence in $(\HH, d)$.
\begin{proof} Convergence in the weak $C^{1,1}$ topology implies uniform
convergence of the potential functions.  Fix
$m,n \in \mathbb N$ and set $\gg_{m,n}(t) = t \phi_n + (1-t) \phi_m$.  We
estimate
\begin{align*}
\EE(\gg_{n,m}) =&\ \int_0^1 \int_M \left( \phi_n - \phi_m \right)^2 \gw^n_{t
\phi_n + (1-t) \phi_m} dt\\
\leq&\ \brs{\phi_n - \phi_m}_{C^0}^2\int_0^1 \int_M \gw^n_{t
\phi_n + (1-t) \phi_m} dt\\
=&\ \brs{\phi_n - \phi_m}_{C^0}^2 V.
\end{align*}
Due to the uniform convergence in $C^0$, the lemma follows.
\end{proof}
\end{lemma}

Furthermore, as mentioned above in Lemma \ref{curvaturecalc}, the space $\HH$
formally has nonpositive curvature.  By again exploiting the theory of
approximate geodesics one can exhibit the nonpositivity of curvature in the
sense of Alexandrov, due to Calabi-Chen \cite{CalabiChen}.  We include a
discussion of the proof to make precise the sense in which it holds with respect
to points in $\bar{\HH}$.

\begin{lemma} \label{triangleineq} (\cite{CalabiChen} Theorem 1.1) Fix $a,b,c
\in \mathcal H$.  For any $s, 0 \leq s \leq 1$, let $p_{s} \in \bar{\HH}$
denote the
point on the geodesic path connecting $b$ to $c$ satisfying $d(b, p_{s}) = s
d(b, c)$ and $d(p_{s}, c) = (1 - s) d(b, c)$.  Then
\begin{align} \label{hyptri}
\bar{d}(a, p_{s})^2 \leq (1 - s) d(a,b)^2 + s d(a, c)^2 - s(1 - s) d(b,
c)^2.
\end{align}
\begin{proof} Let $\gg^{\ge} : [0,1] \to \HH$ denote the $\ge$-approximate
geodesic connecting $b$ to $c$, and let $\EE(s)$ denote the energy of the
$\ge$-approximate geodesic connecting $a$ to $\gg^{\ge}(s)$.
Since the curvature of $\HH$ is nonpositive, Calabi-Chen use Jacobi field
estimates to estimate the second derivative of $\EE(s)$ from below and obtain
the inequality (\cite{CalabiChen} pg. 185)
\begin{align} \label{triineq10}
 \EE(s) \leq (1-s) \EE(0) + s \EE(1) - s(1-s) \left(\EE(\gg^{\ge}) - C \ge
\right).
\end{align}
Fix $s > 0$.  By Lemma \ref{approxgeodlemma} we know that for any sequence
$\{\ge_i\} \to 0$, $\{\gg^{\ge_i}(s)\}$ is a Cauchy sequence in $\HH$, and
moreover $\lim_{i \to \infty} d(b,\gg^{\ge_i}(s)) = s d(b,c)$ and $\lim_{i \to
\infty} d(\gg^{\ge_i}(s),c) = (1-s) d(b,c)$.  In other words this Cauchy
sequence represents the point $p_{s} \in \bar{\HH}$ on the geodesic connecting
$b$ to $c$ Moreover, the energies $\EE(s)$ converge as $\ge \to 0$ to the
squared distance from $a$ to $\gg^{\ge}(s)$, so again by the definition of the
metric on $\bar{\HH}$, the left hand side of (\ref{triineq10}) converges as $\ge
\to 0$ to $\bar{d}(a,p_{s})$.  The lemma follows.
\end{proof}
\end{lemma}

Lastly in this subsection we record Chen's theorem on the decay of $K$-energy
with
distance which is
crucial to what follows.  We give a proof in the simple case of two points
connected by a smooth geodesic, to exhibit why the estimate takes the form it
does.

\begin{thm} \label{Kenergydecay} (\cite{ChenSOKM3} Theorem 1.2) Let $\phi_0,
\phi_1 \in \HH$.  Then
\begin{align*}
\nu(\phi_1) \geq \nu(\phi_0) - d(\phi_0, \phi_1) \sqrt{\mathcal C(\phi_0)}.
\end{align*}
\begin{proof} Fix $\phi_0,\phi_1 \in \HH$, and suppose $\phi : [0, 1] \to \HH$
is
a smooth geodesic connecting $\phi_0$ to $\phi_1$.  We note that
\begin{align*}
\left. \frac{\del \nu}{\del t} \right|_{t = 0} =&\ - \int_M \left(
s_{\phi} -
\bar{s} \right) \dot{\phi} \gw_{\phi}^{\wedge n} \geq - \sqrt{\mathcal
C(\phi_0)} \brs{\brs{\dot{\phi}}}_{\gw_{\phi}} = - \sqrt{\mathcal
C(\phi_0)} d(\phi_0,\phi_1)
\end{align*}
as the geodesic $\phi_t$ has constant speed.  But, by Lemma \ref{Kenergyconv} we
conclude that
for all $t \in [0, 1]$,
\begin{align*}
\frac{\del \nu}{\del t}(t) \geq&\ - \sqrt{\CC(\phi_0)} d(\phi_0,\phi_1).
\end{align*}
Integrating over $[0,1]$ it follows that
\begin{align*}
\nu(\phi_1) \geq \nu(\phi_0) - d(\phi_0,\phi_1) \sqrt{\CC(\phi_0)}.
\end{align*}
The
rigorous proof requires the theory of almost smooth geodesics of Chen-Tian
\cite{ChenTian} discussed in \S \ref{highreggeod}.
\end{proof}
\end{thm}

\subsection{Variational properties}

We will require the variation of geodesic distance as well as a convexity
property for the distance function, which we record below.

\begin{lemma} \label{C1distance} The distance function is $C^1$.  Specifically,
let $\phi(s)$ denote a path in $\mathcal H$ and $L(s)$ denote the geodesic
distance from $0$ to $\phi(s)$,
then
\begin{align*}
\frac{d L}{d s} =&\ E(1,s)^{-\frac{1}{2}} \IP{\IP{\frac{d \phi}{ds}, \frac{\del
\gg(1,s)}{\del t}}} = \left[\int_M \frac{ \del \gg(1, s)}{\del t}
\frac{d \phi}{d s} dV_{s} \right] \left[ \int_M \brs{\frac{\del \gg(1,
s)}{\del t}}^2 dV_{s} \right]^{-\frac{1}{2}}
\end{align*}
where $\gg(t,s) : [0,1] \to \HH$ is the unique $C^{1,1}$ geodesic
from $0$ to $\phi(s)$.
\begin{proof} This is contained in (\cite{ChenSOKM} Theorem 6), but we include
the
proof since the derivative itself is not stated separately therein and moreover
we will use this calculation later.  With $\phi_1(s) \equiv 0$ for all $s$ and
$\phi_2(s) = \phi(s)$ the given path, let $\gg(t,s,\ge)$ denote the family of
approximate geodesics guaranteed by Lemma \ref{approxgeodlemma}.  Furthermore,
set
\begin{align} \label{approxlength}
\mathcal L(s, \ge) := \mathcal L(\gg(\cdot, s, \ge)).
\end{align}
As $\gg(t,s,\ge)$ is a smooth family, we may apply Lemma \ref{lengthvariations}
and the $\ge$-approximate geodesic equation to yield
\begin{align*}
\frac{d \mathcal L(s,\ge)}{d s} =&\ E(t,s,\ge)^{-\frac{1}{2}} \left.
\IP{\IP{\frac{\del \gg}{\del s}, \frac{\del \gg}{\del t}}} \right|_{t=0}^{t=1}\\
&\ \qquad + \int_0^1 E(t,s,\ge)^{-\frac{1}{2}} \IP{\IP{\frac{D}{\del t}
\frac{\del \gg}{\del t}, \frac{\IP{\IP{\frac{\del \gg}{\del s}, \frac{\del
\gg}{\del t}}}}{\IP{\IP{\frac{\del \gg}{\del t}, \frac{\del \gg}{\del t}}}}
\frac{\del \gg}{\del t} - \frac{\del \gg}{\del s}}} dt\\
=&\ E(1,s,\ge)^{-\frac{1}{2}} \IP{\IP{\frac{d \phi}{ds}, \frac{\del
\gg(1,s)}{\del t}}}\\
&\ \qquad + \ge \int_0^1 E(t,s,\ge)^{-\frac{1}{2}} \int_M \left(
\frac{\IP{\IP{\frac{\del \gg}{\del s}, \frac{\del \gg}{\del
t}}}}{\IP{\IP{\frac{\del \gg}{\del t}, \frac{\del \gg}{\del t}}}} \frac{\del
\gg}{\del t} - \frac{\del \gg}{\del s}\right) \Omega dt
\end{align*}
Integrating from $s_1$ to $s_2$ and dividing by $s_2 - s_1$ yields
\begin{align*}
& \brs{\frac{\mathcal L(s_2, \ge) - \mathcal L(s_1, \ge)}{s_2 - s_1} -
\frac{1}{s_2 - s_1} \int_{s_1}^{s_2}  E(s,1,\ge)^{-\frac{1}{2}} \IP{\IP{\frac{d
\phi}{ds}, \frac{\del \gg(1,s)}{\del t}}}}\\
& \qquad \leq\ \frac{\ge}{s_2 - s_1} \int_{s_1}^{s_2} \int_0^1
E(t,s,\ge)^{-\frac{1}{2}} \int_M \left( \frac{\IP{\IP{\frac{\del \gg}{\del s},
\frac{\del \gg}{\del t}}}}{\IP{\IP{\frac{\del \gg}{\del t}, \frac{\del \gg}{\del
t}}}} \frac{\del \gg}{\del t} - \frac{\del \gg}{\del s}\right) \Omega dt ds\\
& \qquad \leq\ C\ge.
\end{align*}
The final estimate follows from Lemma \ref{approxgeodlemma}.  In particular, as
the energy element $E$ approaches a constant it is in particular bounded below,
and the integrands and volume forms are uniformly bounded as well.  Taking the
limit as $\ge$ goes to zero yields
\begin{align*}
\lim_{s_2 \to s_2} \frac{\LL(s_2) - \LL(s_1)}{s_2 - s_1} =&\ \lim_{s_2 \to s_1}
\int_{s_1}^{s_2}  E(s,1)^{-\frac{1}{2}}
\IP{\IP{\frac{\del \gg(1,s)}{\del t}, \frac{d \phi}{ds}}} ds\\
=&\ E(1,s)^{-\frac{1}{2}} \IP{\IP{\frac{d \phi}{ds}, \frac{\del \gg(1,s)}{\del
t}}}\\
=&\ \int_M \frac{\del \gg(1,s)}{\del t} \frac{d \phi}{d s} dV_s \left[ \int_M
\brs{\frac{\del \gg(1,s)}{\del t}}^2 dV_s \right]^{-\frac{1}{2}}.
\end{align*}
\end{proof}
\end{lemma}

\subsection{Higher regularity of geodesics} \label{highreggeod}

In this subsection we recall the improved regularity theory of \cite{ChenTian}. 
We begin by
recalling the interpretation of the geodesic equation in terms of the
homogeneous complex Mong\'e-Ampere (HCMA) equation.

\begin{lemma} Let $\phi : [0,1] \to \HH$ be a continuous path.  Extend this to a
function $\phi : [0,1] \times S^1 \to \HH$ via $\phi(t,\theta,x) = \phi(t)(x)$. 
Then $\phi$ is a geodesic if and only if
\begin{align} \label{HCMA}
\left( \pi_2^* \gw + \sqrt{-1} \del \delb \phi \right)^{n-1} = 0 \quad \mbox{ on
} \Sigma \times M,
\end{align}
where $\Sigma = [0,1] \times S^1$ and $\pi_i$ are the natural projection
operators to $\Sigma$ and $M$.
\end{lemma}

Chen's fundamental existence theorem for $C^{1,1}$ geodesics in $\HH$ is
actually more general, and applies to more general solutions to (\ref{HCMA}).

\begin{thm} (\cite{ChenSOKM} \S 3) For a smooth map $\phi_0 : \del \Sigma \to
\HH$,
there exists a unique $C^{1,1}$ solution $\phi$ of (\ref{HCMA}) such that $\phi
= \phi_0$ on $\del \Sigma$ and $\phi(z,\cdot) \in {\HH^{1,1}}$ for each $z \in
\Sigma$.
\end{thm}

\begin{defn} Suppose $\phi$ is a $C^{1,1}$ solution of (\ref{HCMA}).  The
\emph{regular part} of $\phi$ is
\begin{align*}
\mathcal R_{\phi} = \left\{(z,x) \in \Sigma \times M\ |\ \exists\ U, (z,x) \in
U, \left. \phi \right|_{U} \in C^{\infty}, \mbox{ and } \left. \gw_{\phi}
\right|_{\{z\} \times M} > 0 \right\}.
\end{align*}
Inside $\mathcal R_{\phi}$ we can define a distribution
\begin{align*}
\left. \mathcal D_{\phi} \right|_{(z,x)} = \{ v \in T_z \Sigma \times T_x M\ |\
i_v \left( \pi_2^* \gw + \sqrt{-1} \del \delb \phi \right) = 0 \}.
\end{align*}
Note that since $\gw$ is closed $\mathcal D_{\phi}$ is integrable.  Given a
subset $\VV \subset \Sigma \times M$, we say that $\RR_{\phi}$ is
\emph{saturated} in $\VV$ if every maximal integral submanifold of $\DD_{\phi}$
in $\RR_{\phi} \cap \VV$ is a disk and closed in the subspace topology of $\VV$.
 Note that since $\gw_{\phi} > 0$ in $\RR_{\phi}$, $\DD_{\phi} \cap
T^{1,0}\Sigma \times T^{1,0} M$ is one dimensional, and there is a projection of
a unit length generator of this space onto $T^{1,0}M$.  We call this vector
field $X$.
\end{defn}

\begin{defn} A solution $\phi$ of (\ref{HCMA}) is \emph{partially smooth} if it
satisfies the following conditions:
\begin{enumerate}
\item{It has a uniform $C^{1,1}$ bound on $\Sigma \times M$ and $\RR_{\phi}$ is
saturated in $\Sigma \times M$.}
\item{$\RR_{\phi} \cap (\del \Sigma \times M)$ is open and dense in $\del \Sigma
\times M$.}
\item{The volume form $\gw^n_{\phi}$ can be extended to $\ohat{\Sigma} \times M$
as a continuous $(n,n)$ form, where $\ohat{\Sigma} = \Sigma \backslash \del
\Sigma$.}
\end{enumerate}
\end{defn}

\begin{thm} (\cite{ChenTian} Theorem 1.3.2) Suppose $\Sigma$ is a unit disc. 
Given $\phi_0 : \del \Sigma \to
\HH$ a smooth map, there exists a unique partially smooth solution to
(\ref{HCMA}).
\end{thm}

\begin{defn} A solution $\phi$ of (\ref{HCMA}) is \emph{almost smooth} if it
satisfies the following conditions:
\begin{enumerate}
\item{$\phi$ is partially smooth}
\item{$\DD_{\phi}$ extends to a continuous distribution in an open dense
saturated set $\til{\VV} \subset \Sigma \times M$ such that $\til{\SS}_{\phi} :=
\Sigma \times M \backslash \til{V}$ has the Whitney extension property}
\item{The leaf vector field $X$ is uniformly bounded in $\til{\VV}$.}
\end{enumerate}
\end{defn}

\begin{rmk} The set $\til{\mathcal S}_{\phi}$ is the \emph{singular part} of
$\phi$, and in general $\left( \Sigma \times M \backslash \RR_{\phi} \right)
\backslash \til{\SS}_{\phi} \neq \emptyset$.
\end{rmk}

\begin{thm} (\cite{ChenTian} Theorem 1.3.4) \label{almostsmoothexthm} Suppose
$\Sigma$ is a unit disc.  Given
$\phi_0 : \del \Sigma \to \HH, \phi_0 \in C^{k,\ga}$, $k \geq 2, 0 < \ga < 1$,
and given $\ge > 0$, there exists $\phi_{\ge} : \del \Sigma \to \HH$ and an
almost smooth solution to (\ref{HCMA}) with boundary value $\phi_{\ge}$ such
that
\begin{align*}
\nm{\phi_0 - \phi_{\ge}}{C^{k,\ga}(\del \Sigma \times M)} < \ge.
\end{align*}
\end{thm}

One crucial application of this theory of partially/almost smooth geodesics is
to establish the convexity of the $K$-energy in a more general setting.

\begin{thm} \label{wkKenergyconv} (\cite{ChenTian} Corollary 6.1.2) Suppose
$\phi : \Sigma \to \HH^{1,1}$ is a partially smooth solution to (\ref{HCMA}). 
Then the induced $K$-energy function $\nu : \Sigma \to \mathbb R$ is a
bounded weakly sub-harmonic function in $\Sigma$.
\end{thm}

\noindent This is a deep, technical result which lies at the heart of
Chen-Tian's proof of uniqueness of cscK metrics.  As will become clear in \S
\ref{ltesec} it is crucial to the proof of Theorem \ref{fundexistthm} as well.

\section{Definition of minimizing movements} \label{WKSLNS}

In this section we give the setup for constructing minimizing movement
solutions of $K$-energy.  To begin with we precisely define the functional for
which we are constructing minimizing movements.

\begin{defn} \label{lscextdefn} Let $(X, d)$ be a metric space and $f : X \to
\mathbb R$ a lower
semicontinuous function.  If $(\bar{X}, \bar{d})$ denotes the completion of $(X,
d)$, we define the \emph{lower semicontinuous extension of $f$} by
\begin{align} \label{lscext}
\bar{f}(x) := \begin{cases}
f(x) & x \in X\\
\liminf_{x_n \to x, \{x_n\} \in X} f(x_n) & x \in \bar{X} \backslash X.
\end{cases}
\end{align}
A simple lemma (see Lemma \ref{lscextlsc}
below) shows that $\bar{f}$ is indeed lower semicontinuous.
\end{defn}

\begin{defn} Let $(M^{2n}, \gw, J)$ be a compact K\"ahler manifold.  Recall from
Definition \ref{completiondef} that $\bar{\HH}$ denotes the metric space
completion of $(\HH, d)$.  We set
\begin{align*}
\bar{\nu} : \bar{\HH} \to \mathbb R
\end{align*}
to be the lower semicontinuous extension of $\nu : \HH \to \mathbb R$ in the
sense of Definition \ref{lscextdefn}.
\end{defn}

\begin{rmk} A subtle point in this definition is that, while $\nu$ is
well-defined for $\phi \in \HH^{1,1}$ by Lemma \ref{Kenergform}, it does not
necessarily hold that $\nu(\phi) = \bar{\nu}(\phi)$ for all $\phi \in
\HH^{1,1}$.  The reason for defining $\bar{\nu}$ as the extension of $\nu$ as
defined on $\HH$, instead of $\HH^{1,1}$, is that it is not clear that $\nu$ is
lower semicontinuous on $\HH^{1,1}$.  Indeed, lower semicontinuity for $\nu$ on
$\HH$ follows from Theorem \ref{Kenergydecay}, where the Calabi energy of a
point in $\HH$ controls the rate of increase of $\nu$ approaching that point. 
Certainly one cannot pass this estimate in a naive way to all points in
$\HH^{1,1}$.
\end{rmk}

\begin{defn} \label{discdef} Let $(M^{2n}, \gw, J)$ be a compact K\"ahler
manifold.  Fix ${\phi}
\in \bar{\HH}$ and $\tau > 0$.  Let
\begin{align} \label{Fdef}
\mathcal F_{{\phi}, \tau} ({\psi}) =&\ \frac{\bar{d}^2({\phi}, {\psi})}{2 \tau}
+
\bar{\nu}({\psi}).
\end{align}
Furthermore, set
\begin{align} \label{mudef}
\mu_{\phi,\tau} := \inf_{\psi \in \HH} \FF_{\phi,\tau}(\psi)
\end{align}
The quantity $\mu$ is sometimes referred to as a \emph{Moreau-Yosida
approximation} of the given functional, in this case $\nu$.  Finally, we define
the \emph{resolvent operator}
\begin{align*}
W_{\tau} : \bar{\HH} \to \bar{\HH}
\end{align*}
by the property
\begin{align*}
\FF_{\phi,\tau}(W_{\tau}(\phi)) = \mu_{\phi,\tau}.
\end{align*}
The fact that there exists a unique minimizer for $\FF_{\phi,\tau}$ and so the
map $W_{\tau}$ is well defined will be shown later.
\end{defn}

\noindent Using the resolvent operator we can define our notion of a
discrete Calabi flow.

\begin{defn} \label{wkCF} Let $(M^{2n}, \gw, J)$ be a compact K\"ahler manifold.
 Given $T >
0$, consider a partition of $[0, T)$,
\begin{align*}
0 = t_0 < t_1 < \dots < t_m = T, \quad \tau_i = t_i - t_{i-1}.
\end{align*}
We say that
a sequence $\{\phi_i\}_{i = 0}^m \in \bar{\HH}$ is a \emph{discrete Calabi
flow} with initial condition $\phi_0$ if for all $0 \leq
i \leq m - 1$, $\phi_{i+1} = W_{\tau_i}(\phi_i)$.  We say that the solution has
a \emph{uniform step size $\tau$} if $\tau_i =\tau$
for all $i$.  Associated to any discrete Calabi flow is a one-parameter family
$\phi : [0,T] \to \bar{\HH}$ where $\phi_{|[t_i, t_{i+1})} = \phi_i$.
\end{defn}

\begin{defn} \label{GMM} Let $(M^{2n}, \gw, J)$ be a compact K\"ahler manifold. 
We say that
a curve $\phi : [0,T] \to \bar{\HH}$ is a \emph{$K$-energy minimizing
movement} with initial condition $\phi_0$ if there exists a
sequence of partitions $\{t_i^j\}_{i = 0}^{p_j}$ with associated discrete Calabi
flows $\{\phi_i^j\}_{i=0}^{p_j}$ with initial condition $\phi_0$ as in
Definition \ref{wkCF} such that
\begin{enumerate}
\item{$\lim_{j \to \infty} \sup_i \brs{\tau_i^j} = 0$},
\item{$\forall\ t \in [0, T], \phi^j(t) \to \phi(t)$},
\end{enumerate}
where the convergence above is in the distance topology.
\end{defn}

\begin{rmk} This definition allows for arbitrary step sizes, although the
solutions we construct are convergent limits of discrete Calabi flows with
uniform step sizes (see Theorem \ref{Mayer}).
\end{rmk}

\noindent We proceed to derive some basic properties for $\mathcal F$.  First we
derive the first variation of $\mathcal F$ at smooth points, and compute a
further characterization of its
critical points.  To do this we need a preliminary lemma.

\begin{lemma} \label{scform} Let $(M^{2n}, \gw, J)$ be a compact K\"ahler
manifold.  Fix $\phi
\in \mathcal H$ and $\psi \in C^{\infty}(M)$.  Then
\begin{align*}
\int_M \psi (s_{\phi}  - \bar{s}) \gw_{\phi}^n =&\ \int_M \sqrt{-1} \del \delb
\psi
\wedge \left( - \log \frac{\gw_{\phi}^n}{\gw^n}
\gw_{\phi}^{n-1} \right) + \psi \left( \rho(\gw) - \bar{s} \gw_{\phi}
\right)\wedge \gw_{\phi}^{n-1}.
\end{align*}
\begin{proof} We directly compute
\begin{align*}
\int_M \psi (s_{\phi} - \bar{s}) \gw_{\phi}^n =&\ \int_M \psi \left(
\rho(\gw_{\phi}) - \bar{s} \gw_{\phi} \right) \wedge \gw_{\phi}^{n-1}\\
=&\ \int_M \psi \left(\rho(\gw) - \frac{\sqrt{-1}}{2}
\del
\delb \log \frac{\gw_{\phi}^{n}}{\gw^n} - \bar{s} \gw_{\phi} \right) \wedge
\gw_{\phi}^{n-1}\\
=&\ \int_M \sqrt{-1} \del \delb \psi
\wedge \left( - \log \frac{\gw_{\phi}^n}{\gw^n}
\gw_{\phi}^{n-1} \right) + \psi \left( \rho(\gw) - \bar{s} \gw_{\phi}
\right)\wedge \gw_{\phi}^{n-1}.
\end{align*}
\end{proof}
\end{lemma}

\begin{lemma} \label{Ffirstvarlemma} Let $(M^{2n}, \gw, J)$ be a compact
K\"ahler
manifold.  Fix $\phi
\in \HH$ and $\psi_s \in \mathcal H$ a one-parameter family of functions, $s \in
(-\ge, \ge)$.  Then
\begin{align} \label{Ffirstvar}
\left. \frac{d}{ds} \mathcal F_{\phi,\tau}(\psi_s) \right|_{s=0} =&\ \IP{\IP{
\frac{1}{\tau} \frac{\del \gg(1,0)}{\del t} - s_{\psi} + \bar{s}, \frac{d
\psi}{d s}}}
\end{align}
where $\gg : [0,1] \to \HH$ is the unique $C^{1,1}$ geodesic connecting $\phi$
to $\psi_0$.
Furthermore, $\phi \in \HH$ is a critical point for $\FF_{\phi,
\tau}$ if and only if for all $\eta \in C^{\infty}(M)$,
\begin{align} \label{discreteEuler}
0 =&\ \int_M \left[ \eta \left( \frac{1}{\tau} \frac{\del \gg(1,0)}{\del t} +
\bar{s} \right) \gw_{\phi} - \eta \rho(\gw) + \log \frac{\gw_{\phi}^n}{\gw^n}
\sqrt{-1} \del \delb \eta \right] \wedge \gw_{\phi}^{n-1}.
\end{align}
\begin{proof} Using Lemma \ref{lengthvariations} we compute the variation
\begin{align*}
\left. \frac{d}{ds} \frac{d^2(\phi,\psi_s)}{2 \tau} \right|_{s=0} =&\
\frac{d(\phi, \psi)}{\tau E(1,0)^{\frac{1}{2}}} \IP{\IP{ \frac{\del
\gg(1,0)}{\del t}, \frac{d \psi}{d s}}}.
\end{align*}
But since $\gg$ has constant speed, in particular we have $E(1,0)^{\frac{1}{2}}
= d(\phi,\psi)$.  Combining this with the definition of $K$-energy and Lemma
\ref{scform} yields the result.
\end{proof}
\end{lemma}

\begin{lemma} Let $(M^{2n}, \gw, J)$ be a compact K\"ahler manifold, and fix
$\phi \in \HH$, $\tau > 0$.  If $\psi$ is a minimizer for $\FF_{\phi,\tau}$,
then $I(\psi) = I(\phi)$.
\begin{proof} Without loss of generality we assume $\phi = 0$, the general case
being analogous.  Let $\psi_s = \psi + s$. Certainly $\psi_s$ is a geodesic. 
Moreover, it is clear by construction that $\nu(\psi_s) = \nu(\psi)$
for all $s$.  Since $\psi$ realizes the minimum for $\FF_{\phi,\tau}$, it thus
follows that $\psi$ realizes the minimum for $d(\phi, \psi_s)$ in the variable
$s$.  Let $\gg$ denote the unique $C^{1,1}$ geodesic connecting $0$ to $\phi$. 
By Lemma \ref{lengthvariations}, varying through the curve $\psi_s$, it follows
that
\begin{align*}
0 =&\ \left. \IP{\IP{1, \frac{\del \gg}{\del t}}}_{\gg} \right|_{t=1} = \left.
\frac{d}{dt} I(\gg_t) \right|_{t=1}.
\end{align*}
But since $\gg$ is a geodesic we conclude that
\begin{align*}
\frac{d^2}{dt^2} I(\gg_t) =&\ \frac{d}{dt} \IP{\IP{1, \frac{\del \gg}{\del t}}}
= \IP{\IP{1, \frac{D}{\del t} \frac{\del \gg}{\del t}}} = 0.
\end{align*}
Thus $\frac{d}{dt} I(\gg_t) = 0$ for all $t$, and hence $I(\gg_1) = I(\gg_0)$.
\end{proof}
\end{lemma}

\section{Long time existence of minimizing movements} \label{ltesec}

In this section we prove Theorem \ref{fundexistthm}.
As mentioned in the introduction, the proof will be an application of an 
theorem of Mayer \cite{Mayer} on the long time existence of minimizing
movements for convex lower semicontinuous functionals on complete nonpositively
curved metric spaces (\cite{Mayer} Theorem 1.13, Theorem \ref{Mayer} below).  We
begin by recalling the setup and statement of this theorem.

\subsection{Mayer's Theorem}

\begin{defn} A metric space $(X, d)$ is a \emph{path-length space} if any two
points $x,y \in X$ can be connected by a path $\gg :[0,1] \to X$ such
that for all $t \in [0,1]$, $d(x,\gg(t)) = t d(x,y)$.  Such a path $\gg$ will be
called a \emph{constant speed geodesic}.
\end{defn}

\begin{defn} We say that a metric space $(X, d)$ is an \emph{NPC space} if $(X,
d)$ is a path-length space and for every choice of points $a,b,c \in X$, if $\gg
: [0,1] \to X$ denotes the unique geodesic path connecting $b$ to $c$, then for
all $t \in[0,1]$ one has
\begin{align} \label{hyptri2}
d(a,\gg(t))^2 \leq (1 - t) d(a,b)^2 + t d(a,c)^2 - t(1-t) d(b,c)^2.
\end{align}
\end{defn}

\noindent Inequality (\ref{hyptri2}) is a kind of ``hyperbolic triangle
inequality'' in that it holds on smooth manifolds with nonpositive curvature and
intuitively says that triangles bend inwards.  In \S \ref{furtherprops} we will
furthermore require the analogous ``quadrilateral comparison'' inequality for
NPC spaces.

\begin{thm} \label{quadcomp} (\cite{Korevaar} Corollary 2.1.3) Let $(X, d)$ be a
complete NPC
space.  Given $x_0,x_1,y_0,y_1 \in X$, let $x_t$ and $y_t$ denote the geodesics
connecting $x_0$ to $x_1$ and $y_0$ to $y_1$ respectively.  Then for all $t \in
[0,1]$ one has
\begin{align*}
d^2(x_t,y_0) + d^2(x_{1-t},y_1) \leq&\ d^2(x_0,y_0) + d^2(x_1,y_1) + 2 t^2
d^2(x_0,x_1)\\
&\ + t ( d^2(y_0,y_1) - d^2(x_0,x_1)) - t(d(y_0,y_1) - d(x_0,x_1))^2. 
\end{align*}
\end{thm}

\begin{defn} Let $(X, d)$ be an NPC space.  Given $B \geq 0$, a function $f:X
\to \mathbb R$ is \emph{$B$-convex} if for all $x_0,x_1 \in X$, we let $x_t :
[0,1] \to X$ denotes the geodesic connecting $x_0$ to $x_1$, then one has
\begin{align} \label{Bconv}
f(x_t) \leq (1-t) f(x_0) + t f(x_1) + Bt(1-t) d^2(x_0,x_1).
\end{align}
for all $t \in [0,1]$.
\end{defn}

\begin{thm} \label{Mayer} (\cite{Mayer} Theorem 1.13) Let $(X, d)$ be a complete
NPC space, and let $f : X \to (-\infty, \infty]$ satisfy
\begin{enumerate}
\item{$f$ is lower semicontinuous,}
\item{$f$ is $B$-convex for some $B \geq 0$.}
\end{enumerate}
Fix $y \in X$ and let
\begin{align*}
A :=&\ - \min \left\{0, \liminf_{d(x,y) \to \infty} \frac{f(x)}{d^2(x,y)}
\right\},\\
I_A :=&\ \begin{cases}
(0,\infty) & \mbox{ for } A = 0,\\
\left(0, \frac{1}{16 A} \right] & \mbox{ for } A > 0.
\end{cases}
\end{align*}
Then given $x_0 \in X$ with $f(x_0) < \infty$, there exists a function $x : I_A
\to X$ satisfying
\begin{align} \label{solnprop1}
x(t) =&\ \lim_{n \to \infty} W_{\frac{t}{n}}^n(x_0),\\  \label{solnprop2}
f(x(t)) \leq&\ f(x_0) \mbox{ for all } t \in I_A,\\ \label{solnprop3}
\lim_{t \to 0} x(t) =&\ x_0. 
\end{align}
Furthermore, the convergence in (\ref{solnprop1}) is uniform on compact
subintervals of $I_A$.
\end{thm}

\begin{rmk} The operators $W^n_{\frac{t}{n}}$ are iterations of the resolvent
operator as in Definition \ref{discdef}.  Note that the theorem first of all
asserts the long time existence of discrete gradient flows, in the sense of
Definition \ref{discdef}, with any initial condition and arbitrarily small
uniform step size.  Moreover, it asserts that a sequence of such converges in
the distance topology to some limiting path through $X$, which by definition is
a minimizing movement.
\end{rmk}

\begin{rmk} Of course implicit in the statement of the theorem is that for every
$t \in [0, \infty)$, the sequence $\{W^n_{\frac{t}{n}}(x_0)\}$ is Cauchy in
$(X, d)$, and so in particular lies in a ball of some controlled
size around $0$.  However, it is relevant to the proof of Theorem \ref{highreg}
to precisely exhibit the dependence of this distance on the initial data.  This
is included in the work of Mayer, and for convenience we include the technical
lemma below.  Moreover, this lemma makes clear the role of the constant $A$,
which as a measure of the decay rate of $f$ in turn controls how fast the
distance of points along the flow can grow.
\end{rmk}

\begin{lemma} (\cite{Mayer} Lemma 1.11) \label{distcntrl} Let $(X, d)$ be a
complete NPC space, and let $f : X \to (-\infty, \infty]$ satisfy the conditions
of Theorem \ref{Mayer}.  Let $x_0 \in X$ satisfy $f(x_0) < \infty$, and let
$x_{j+1} = W_h(x_j)$.  If $A$ as defined in Theorem \ref{Mayer} satisfies $A =
0$, then for $T > 0$, one has
\begin{align*}
 d^2(x_0, x_j) \leq B j h,
\end{align*}
where $B$ depends on $f(x_0), d(x_0, y)$, $T$, and $\inf_{B_{R_T}(y)} f(x)$,
where $R_T$ is chosen so that for all $x$ satisfying $d(x,y) \geq R$, one has
$f(x) \geq
-\frac{d^2(x,y)}{8T}$.
\begin{proof} By the definition of the resolvent operator, we have that
$\FF_{x_i,h}(x_{i+1}) \leq \FF_{x_i,h}(x_i)$, or in other words
\begin{align} \label{distcntrl10}
 \frac{1}{2 h} d^2(x_{i+1},x_i) \leq&\ f(x_i) - f(x_{i+1}).
\end{align}
 Using the triangle inequality and Cauchy-Schwarz yields
\begin{align} \label{distcntrl20}
 d^2(x_0, x_j) \leq&\ \left( \sum_{i=0}^{j-1} d \left( x_i, x_{i+1} \right)
\right)^2 \leq j \sum_{i = 0}^{j-1} d^2(x_i, x_{i+1}).
\end{align}
Combining (\ref{distcntrl10}) and (\ref{distcntrl20}) yields
\begin{align*}
 d^2(x_0, x_j) \leq&\ 2 j h \left( f(x_0) - f(x_j) \right).
\end{align*}
Now choose $K > 0$ so that for all $x \in X$ one has
\begin{align*}
 f(x) \geq - K - \frac{d^2(x, y)}{8T}.
\end{align*}
Note that $K$ can be chosen to be $K = \min \left\{1, - \inf_{B_R(y)} f
\right\}$ where $R$ is
chosen so that for all $x$ satisfying $d(x,y) \geq R$, one has $f(x) \geq
-\frac{d^2(x,y)}{8T}$, which exists by our hypothesis on $A$.
Thus we have
\begin{align*}
 d^2(x_0, x_j) \leq 2 j h \left( f(x_0) + K + \frac{1}{4T} (d^2(x_j, x_0) +
d^2(x_0,y)) \right).
\end{align*}
Since $jh \leq T$ we conclude that
\begin{align*}
 d^2(x_0,x_j) \leq 4 j h \left( f(x_0) + K + \frac{1}{4T} d^2(x_0,y) \right).
\end{align*}
\end{proof}
\end{lemma}

\subsection{Proof of Theorem \ref{fundexistthm}}
\subsubsection{NPC Property}

\begin{lemma} \label{NPCcompl} Let $(X, d)$ be a metric space such that every
pair of points $x,y \in X$ is connected by a unique geodesic $\gg : [0,1] \to
(\bar{X}, \bar{d})$, and that the hyperbolic triangle inequality (\ref{hyptri})
holds for triples of points $a,b,c \in X$.  Then $(\bar{X}, \bar{d})$ is an NPC
space.

\begin{proof} The proof is summarized in Figure \ref{fig1}.   Fix $\bar{x},
\bar{y}$ two points in $\bar{X}$, represented by
Cauchy sequences $\{x_n\}, \{y_n\}$.  Let $\gg_n : [0,1] \to \bar{X}$
denote the unique constant speed geodesic connecting $x_n$ to $y_n$.  In
particular, one has $\bar{d}(x_n, \gg_n(s)) = s \bar{d}(x_n, y_n)$ for all $n
\in \mathbb N, s \in [0,1]$.  Let $l_n := \bar{d}(x_n, y_n)$.  By the definition
of
the distance function on the completion one has that $\lim_{n \to \infty} l_n =
\bar{d}(x, y) = l$.  Fix some $s \in (0,1)$.  We claim that $\{\gg_n(s) \}$ is a
Cauchy sequence.  Given $\ge > 0$, first choose $N$ large so that for all $n,m >
N$, $\bar{d}(x_m, x_n) < \ge$, $\bar{d}(y_m, y_n) < \ge$.  Furthermore, fix a
point $\til{\gg}_n(s) \in X$ such that $d(\gg_n(s), \til{\gg}_n(s)) < \ge$. 
Now consider the
triangle $\Delta
abc$ with $a = \til{\gg}_n(s), b = x_m, c = y_m$.  We may apply (\ref{hyptri})
to
yield
\begin{align*}
\bar{d}(\til{\gg}_n(s), \gg_m(s))^2 \leq (1 - s) d(\til{\gg}_n(s), x_m)^2 + s
d(\til{\gg}_n(s), y_m)^2 -
s(1 - s) {d}(x_m, y_m)^2.
\end{align*}
But now note
\begin{align*}
{d}(\til{\gg}_n(s), x_m) \leq&\ \bar{d}(\til{\gg}_n(s), \gg_n(s)) +
\bar{d}(\gg_n(s), x_n) + \bar{d}(x_n, x_m)\\
\leq&\ s l_n + 2 \ge.
\end{align*}
Likewise
\begin{align*}
{d}(\til{\gg}_n(s), y_m) \leq&\ (1-s) l_n + 2 \ge.
\end{align*}
\begin{figure}[h] 
\begin{tikzpicture}[scale=3]

\draw [fill=black] (-4.7,2.7) node [left] (v1) {\tiny{$x_n$}} circle (0.03);
\draw [fill=black] (-4.3,2.1) node [left] (v3) {\tiny{$x_m$}} circle (0.03);
\draw [fill=black] (-4.1,1.8) node [left] (v5) {\tiny{$\overline{x}$}} circle
(0.03);
\draw [fill=black] (-3.7616,3.232) node [above left] (v7)
{\tiny{$\til{\gamma}_n(s)$}} circle (0.03);
\draw [fill=black] (-3.4907,2.4374) node [below] (v8) {\tiny{$\gamma_m(s)$}}
circle (0.03);
\draw [fill=black] (-3.336,2.1175) node [below] (v9)
{\tiny{$\overline{\gamma}(s)$}} circle (0.03);

\draw [fill=black] (-3.9226,3.0337) node [above left] (v1)
{\tiny{$\gamma_n(s)$}} circle (0.03);

\draw [fill=black] (-1.4,4.1) node [right] (v2) {\tiny{$y_n$}} circle (0.03);
\draw [fill=black] (-1,3.5) node [right] (v4) {\tiny{$y_m$}} circle (0.03);
\draw [fill=black] (-0.8,3.2) node [right] (v6) {\tiny{$\overline{y}$}} circle
(0.03);

\node (v7) at (-0.7455,3.2194) {};
\node (v8) at (-0.9353,3.5147) {};
\node (v9) at (-1.3488,4.1204) {};

\node (v10) at (-4.1405,1.7784) {};
\node (v11) at (-4.342,2.0807) {};
\node (v12) at (-4.7466,2.6841) {};

\node (v13) at (-3.7779,3.2878) {};
\node (v14) at (-3.4696,2.3811) {};

\draw  [dashed] (v7) edge (v10);
\draw  (v8) edge (v11);
\draw  (v9) edge (v12);
\draw [dashed] (v13) edge (v14);

\draw (-4.2969,2.1004) .. controls (-3.9876,2.2902) and (-3.7134,2.958) ..
(-3.7616,3.225);

\draw (-0.9999,3.4923) .. controls (-1.9067,3.1971) and (-3.2635,2.9299) ..
(-3.7546,3.225) node (v13) {};

\node at (-4.1077,2.6119) {\tiny{$\approx s l_n$}};
\node at (-2.6295,3.2634) {\tiny{$\approx (1-s) l_n$}};

\node at (-3.454,2.7643) {\tiny{$\approx 0$}};
\end{tikzpicture}
\caption{Convergence of geodesics in $\bar{\HH}$}
\label{fig1}
\end{figure}

\noindent Combining these facts yields
\begin{align*}
\bar{d}(\gg_n(s), \gg_m(s))^2 \leq&\ \bar{d}(\gg_n(s), \til{\gg}_n(s))^2 +
\bar{d}(\til{\gg}_n(s), \gg_m(s))^2\\
\leq&\ (1-s) (s l_n + \ge)^2 + s ( (1-s) l_n +
\ge)^2 -
s(1-s) l_m^2 + C(s,l)\ge\\
=&\ s(1-s) \left( s l_n^2 + (1-s) l_n^2 - l_m^2 \right) + C(s,l) \ge\\
=&\ s(1-s) \left( l_n^2 - l_m^2 \right) + C(s,l) \ge.
\end{align*}
Since $\lim_{n \to \infty} l_n = l$, the claim follows.  In particular, since
$\{\gg_n(s) \}$ is Cauchy for every $s$, this defines the limiting curve
$\bar{\gg}$.  It is clear that the property $\bar{d}(\bar{x}, \bar{\gg}(s)) = s
l$
passes to the limit.  Furthermore, given that we have established that geodesics
connecting Cauchy sequences of points
in $X$ converge to a limiting geodesic curve connecting two points in $\bar{X}$,
property (\ref{hyptri}) certainly passes to the limit as well. 
Also, uniqueness of the limiting geodesic follows by a very similar argument to
the existence.  In particular if $\til{\gg} :[0,1] \to \bar{X}$ is another curve
satisfying $d(x,\til{\gg}(s)) = s d(x,y)$ and $d(\til{\gg}(s),y) = (1-s) d(x,y)$
for all $s$, then by using the hyperbolic triangle inequality for the triangle
$\gD xy\til{\gg}(s)$ one obtains directly that $\til{\gg}$ is the same as the
curve constructed above.
\end{proof}
\end{lemma}

\subsubsection{Lower semicontinuity}

\begin{lemma} \label{lscextlsc} The function $\bar{f}$ as defined in
(\ref{lscext}) is lower semicontinuous.
\begin{proof} Fix $x_{\infty} \in \bar{X}$, and choose a sequence
$\{x_i \in \bar{X}\}$ converging to $x_{\infty}$ in the distance
topology.  By a diagonalization argument, for each $x_i$ we may choose a
sequence $x_i^j \in X$ such that
\begin{align*}
 \lim_{j \to \infty} d(x_i^j, x_i) =&\ 0\\
 \lim_{j \to \infty} f(x_i^j) =&\ \bar{f}(x_i).
\end{align*}
Note that if $x_i \in X$ one simply chooses $x_i^j = x_i$ for all $j$.  For each
$i$ choose $N_i$ such that for all $j \geq
N_i$, $d(x_i^j, x_i) \leq \frac{1}{i}$ and $f(x_i^j) \leq \bar{f}(x_i) +
\frac{1}{i}$.  It
follows from the triangle inequality that $\lim_{i \to \infty}
d(x_i^{{N}_i}, x_{\infty}) = 0$.  If $x_{\infty} \in X$ then since $f$ is lower
semicontinuous on $X$ one has
\begin{align*}
\bar{f}(x_{\infty}) = f(x_{\infty}) \leq \lim_{i \to \infty} f(x_i^{{N}_i})
\leq&\ \lim_{i \to \infty} \bar{f}(x_i) + \frac{1}{i} = \lim_{i \to \infty}
\bar{f}(x_i).
\end{align*}
On the other hand, if $x_{\infty} \in \bar{X} \backslash X$ then by definition
we have
\begin{align*}
\bar{f}(x_{\infty}) \leq \lim_{i \to \infty} f (x_i^{{N}_i})
\leq&\ \lim_{i \to \infty} \bar{f}(x_i) + \frac{1}{i} = \lim_{i \to \infty}
\bar{f}(x_i).
\end{align*}
The lemma follows.
\end{proof}
\end{lemma}

\begin{lemma} \label{Kenerglsc} The function $\nu$ is lower semicontinuous on
$(\HH, d)$.
\begin{proof} This follows directly from Theorem \ref{Kenergydecay}.
\end{proof}
\end{lemma}

\begin{lemma} For every $\phi \in \bar{\HH}$,
$\bar{\nu}(\phi) > -\infty$.
 \begin{proof} Fix $\phi \in \bar{\HH}$, and $\{\phi_i\} \in \HH$ any
sequence converging to $\phi$ in the distance topology.  By the triangle
inequality we have, for sufficiently large $i$,
\begin{align*}
 d(0, \phi_i) \leq d(0, \phi) + d(\phi, \phi_i) \leq C.
\end{align*}
Thus by Theorem \ref{Kenergydecay} there exists a constant $C$ depending only on
$\phi$ such that
\begin{align*}
 \lim_{i \to \infty} \nu(\phi_i) \geq - C.
\end{align*}
We conclude that $\bar{\nu}(\phi) > - \infty$.
\end{proof}
\end{lemma}

\subsubsection{Convexity}

In this subsection we establish geodesic convexity of $\bar{\nu}$ on
$\bar{\HH}$.  We begin with some preliminary lemmas.

\begin{lemma} \label{Kenergapprox} Given $\phi \in \HH^{1,1}$ there exists a
sequence $\{\phi_i\} \in \HH$ converging to $\phi$ in the weak $C^{1,1}$
topology such that
\begin{align*}
 \lim_{i \to \infty} \nu(\phi_i) = \nu(\phi).
\end{align*}
\begin{proof} By convolving with a mollifier we can construct a sequence
$\{\phi_i\}$ with a uniform $C^{1,1}$ bound converging strongly to $\phi$ in
$C^{1,\ga}$ and $W^{2,p}$,
and moreover $\frac{\gw_{\phi_i}^n}{\gw^n}$ converges strongly to
$\frac{\gw_{\phi}^n}{\gw^n}$ in $L^2$.  Due to this strong
convergence, there exists a subsequence, still denoted $\{\phi_i\}$, such that
$\frac{\gw_{\phi_{i}}^n}{\gw^n}$ converges to $\frac{\gw_{\phi}^n}{\gw^n}$
almost everywhere.

First note that by the strong convergence in $C^{1,\ga}$ it follows directly
from
(\ref{Iformula2}) and (\ref{Jformula2}) that the $I$ and $J$ terms appearing in
(\ref{Kenergform}) converge along this sequence to the limiting value at $\phi$.
 To deal with the
remaining term, we observe that almost everywhere convergence of the volume form
ratios implies that
$f_k :=
\frac{\gw_{\phi_{t_k}}^n}{\gw^n} \log \frac{\gw_{\phi_{t_k}}^n}{\gw^n}$
converges almost
everywhere to $f := \frac{\gw_{\phi}^n}{\gw^n} \log
\frac{\gw_{\phi}^n}{\gw^n}$.  By a standard measure theoretic lemma, since $M$
is compact, and in particular the measure induced by $\gw$ is finite, we may
choose a further subsequence to obtain $f_k
\to f$ almost everywhere and with respect to measure.  Note that the sequence
$\{f_k\}$ satisfies uniform upper and lower bounds depending on the $C^{1,1}$
bound for $\phi$, and we call this bound $A$.  By the
convergence in
measure, for any $\ge > 0$ we may
choose $N_{\ge}$ large so that for all $n \geq N$, $U_k := \{x \in M | \brs{f_k
- f} \geq \ge\}$ satisfies $\int_{U_k} \gw^n \leq \ge$.  We then have, for all
$k \geq N_{\ge}$,
\begin{align*}
 \brs{\int_M \left(f - f_k \right) \gw^n} \leq&\ \int_M \brs{f - f_k} \gw^n\\
=&\ \int_{U_k} \brs{f - f_k} \gw^n + \int_{M \backslash U_n} \brs{f - f_k}
\gw^n\\
\leq&\ 2 A \ge + V \ge.
\end{align*}
It follows that the sequence $\{\phi_{N_{\frac{1}{n}}} \}$ satisfies the
required properties.
\end{proof}
\end{lemma}

\begin{lemma} \label{C11lsc} Given $\{\phi_i\} \in \HH^{1,1}$ such that
$\{\phi_i\} \to \phi \in \HH^{1,1}$ in $C^{1,\ga}$, one has
\begin{align*}
\bar{\nu}(\phi) \leq \liminf_{i \to \infty} \nu(\phi_i).
\end{align*}
 \begin{proof} By Lemma \ref{Kenergapprox} we may choose sequences
$\{\phi_i^j\} \in \HH$ converging to $\phi_i$ in the weak $C^{1,1}$ topology
such that
\begin{align*}
 \lim_{j \to \infty} \nu(\phi_i^j) = \nu(\phi_i).
\end{align*}
Since convergence in the weak $C^{1,1}$ topology implies convergence in
$C^{1,\ga}$, which implies convergence in the
distance topology, a simple diagonalization argument yields a sequence
$\phi_i^{j_i}$ converging to $\phi$ in the distance topology, such that
\begin{align*}
 \lim_{i \to \infty} \nu(\phi_i^{j_i}) = \liminf_{i \to \infty} \nu(\phi_i).
\end{align*}
The result follows from the definition of $\bar{\nu}$.
\end{proof}
\end{lemma}

\begin{lemma} \label{convext} Given the setup of Lemma \ref{NPCcompl}, let $f :
X \to \mathbb R$ be lower semicontinuous.  If $\bar{f}$ is convex along
geodesics connecting points in $X$, then
$\bar{f}: \bar{X} \to \mathbb R$ is geodesically convex.
\begin{proof} Fix $\bar{x}, \bar{y} \in \bar{X}$, and fix $x_n \to \bar{x}$
a Cauchy sequence in $X$ such that
\begin{align*}
\lim_{n \to \infty} f(x_n) = \bar{f}(\bar{x}).
\end{align*}
Likewise define $y_n$.  These sequences exist by the definition of $\bar{f}$. 
Let $\gg_n : [0,1] \to \bar{X}$ denote the unique geodesic connecting $x_n$ to
$y_n$. 
By the assumed convexity of $\bar{f}$ on geodesics connecting points in $X$ one
obtains
\begin{align*}
\bar{f}(\gg_n(t)) \leq (1-t) \bar{f}(x_n) + t \bar{f}(y_n).
\end{align*}
Taking the limit yields
\begin{align*}
\lim_{n \to \infty} \bar{f}(\gg_n(t)) \leq&\ (1-t) \bar{f}(\bar{x}) + t
\bar{f}(\bar{y}).
\end{align*}
But as in the proof of Lemma \ref{NPCcompl}, we know that $\lim_{n \to \infty}
\gg_n(t) = \bar{\gg}(t)$, where $\bar{\gg} : [0,1] \to \bar{X}$ denotes the
geodesic connecting $\bar{x}$ to $\bar{y}$.  Therefore by lower semicontinuity
of $\bar{f}$ we obtain
\begin{align*}
\bar{f}(\bar{\gg}(t)) \leq&\ \lim_{n \to \infty} \bar{f}(\gg_n(t)) \leq (1-t)
f(\bar{x}) + t
f(\bar{y}).
\end{align*}
Thus $\bar{f}$ is geodesically convex.
\end{proof}
\end{lemma}

\begin{prop} \label{convexityprop} $\bar{\nu} : \bar{\HH} \to \mathbb R$ is
geodesically convex, that
is, given $\phi : [0,1] \to \bar{\HH}$ a geodesic, for all $t \in [0,1]$ one
has
\begin{align} \label{convexity}
 \bar{\nu}(\phi_t) \leq (1-t) \bar{\nu}(\phi_0) + t \bar{\nu}(\phi_1).
\end{align}
\begin{proof} The main step is to show that (\ref{convexity}) holds for $\phi_0,
\phi_1 \in \HH$, in which case
$\phi_t \in \HH^{1,1}$ is the unique $C^{1,1}$ geodesic connecting $\phi_0$ to
$\phi_1$.  The proposition
then follows from Lemma \ref{convext}.  We prove this fact using the ``oval
approximations'' of $C^{1,1}$ geodesics as
introduced in \cite{ChenTian}.  Consider the surface with boundary in the plane
$\mathbb C$,
\begin{align*}
\Sigma^l := [-l,l] \times [0,1] \cup D_{\frac{1}{2}}(-l, \tfrac{1}{2}) \cup
D_{\frac{1}{2}} (l, \tfrac{1}{2})
\end{align*}
where $D_r(s,t)$ denotes the Euclidean disc of radius $r$ around the point
$(s,t)$.  The set $\Sigma^l$ is a ``stadium'' shape, whose boundary consists of
the two lines $(-l,l) \times \{0\}, (-l,l) \times \{1\}$ and two semicircles
denoted $C_{\pm}$.  Also we set $\Sigma^{\infty} := (-\infty,\infty) \times
[0,1]$.  Fix diffeomorphisms $\gg_{\pm} : [0,1] \to C_i$ such that
$\gg_{\pm}(0) = (\pm l,0), \gg_{\pm}(1) = (\pm l,1)$.  Given $\phi_0, \phi_1 \in
\HH$ consider the map
\begin{align*}
\phi^{l} &: \Sigma^l \to \HH\\
&: (0,t) \to \phi_0\\
&: (1,t) \to \phi_1\\
&: (s,t) \to \left( \gg_{\pm}^{-1}(s,t) \right) \psi_1 + \left (1 -
\gg_{\pm}^{-1}(s,t) \right) \psi_0 \qquad \mbox{ for } (s,t) \in C_{\pm}.
\end{align*}
In particular, the boundary value on the two semicircles smoothly interpolates
between $\phi_0$ and $\phi_1$.  Let $\psi^l : \del
\Sigma^{l} \to \HH$ denote the boundary map of $\phi^l$.  Furthermore, for any
$\gd > 0$, by Theorem
\ref{almostsmoothexthm}
there exists a boundary map $\psi^{l,\gd} : \del \Sigma^{l} \to \HH$ and an
almost smooth geodesic $\phi^{l,\gd} : \Sigma^{l} \to \HH$ with this boundary
condition, such that
\begin{align*}
 \brs{\phi^{l,\gd}}_{C^{1,1}} \leq&\ A, \qquad \lim_{\gd \to 0} \max_{\del
\Sigma^{l} \times M} \nm{ \psi^{l,\gd} -
\psi^{l}}{C^{2,\ga}} = 0.
\end{align*}
The constant $A$ above is independent of both $l$ and $\gd$, so in particular
the $K$-energies in the image of $\phi^{l,\gd}$ are uniformly bounded as well.
Let $\nu^{l,\gd}(s,t) := \nu(\phi^{l,\gd}(s,t))$.  By Theorem
\ref{wkKenergyconv}
$\nu^{l,\gd}$ is weakly subharmonic.  Furthermore let
$f^{l,\gd}(s,t) := \nu^{l,\gd}(s,t) - (1-t) \nu(\phi_0) - t
\nu(\phi_1)$.  Certainly $f^{l,\gd}$ is also weakly subharmonic.

Now fix $\gk : \mathbb R \to \mathbb R$ a smooth nonnegative cutoff function
such that
\begin{align*}
\gk \equiv 1 \mbox{ on }& {[-\tfrac{1}{2}, \tfrac{1}{2}]}, \qquad \supp \gk
\subset [-\tfrac{3}{4}, \tfrac{3}{4}].
\end{align*}
Furthermore set
\begin{align*}
 \gk^{m}(s) =&\ \frac{\gk \left( \frac{s}{m} \right)}{\bar{\gk}}, \qquad
\bar{\gk} := \int_{-\infty}^{\infty} \gk(s) ds.
\end{align*}
The set up is summarized in Figure \ref{fig2}.

\begin{figure}[h]
\begin{tikzpicture}[scale=1.5]

\node (v1) at (-5.6,1.4) {};
\node (v2) at (1.81,1.4) {};
\draw  (v1) edge (v2);

\node (v3) at (-5.6,3.0) {};
\node (v4) at (1.81,3.0) {};
\draw  (v3) edge (v4);

\draw (-5.47,3.0) arc [radius=0.8, start angle=90, end angle = 270];
\draw (1.68,1.4) arc [radius=0.8, start angle=-90, end angle = 90];

\node (v5) at (2.4275,1.1907) {};
\node (v6) at (2.9536,1.1874) {};
\draw  [<->](v5) edge (v6);
\node at (2.6853,1.0735) {\tiny{$s$}};

\node (v7) at (2.9618,1.7625) {};
\node (v8) at (2.9571,1.186) {};
\draw [<->] (v7) edge (v8);
\node at (3.0837,1.4625) {\tiny{$t$}};

\node (v9) at (-4,3.12) {};
\node (v10) at (-4,1.28) {};
\node (v11) at (0.2,3.12) {};
\node (v12) at (0.2,1.28) {};

\draw (v9) edge (v10);
\draw  (v11) edge (v12);
\node at (-3.5,2.8) {\tiny{$\supp \kappa^m$}};

\node at (-3.991,1.23) {\tiny{$(-m,0)$}};
\node at (-5.52,3.16) {\tiny{$(-l,1)$}};
\node at (-5.52,1.23) {\tiny{$(-l,0)$}};
\node at (-4.0262,3.16) {\tiny{$(-m,1)$}};

\node at (0.1355,1.23) {\tiny{$(m,0)$}};
\node at (1.6891,3.16) {\tiny{$(l,1)$}};
\node at (1.6821,1.23) {\tiny{$(l,0)$}};
\node at (0.1707,3.16) {\tiny{$(m,1)$}};

\node at (-1.7907,1.075) {\tiny{$\phi^l \equiv \phi_0$}};
\node (v13) at (-2.2,1.25) {};
\node (v14) at (-1.4,1.25) {};
\draw  [<->] (v13) edge (v14);

\node at (-1.7907,3.3) {\tiny{$\phi^l \equiv \phi_1$}};
\node (v15) at (-2.2,3.15) {};
\node (v16) at (-1.4,3.15) {};
\draw  [<->] (v15) edge (v16);

\draw [fill=black] (-5.51,1.40) circle (0.02);
\draw [fill=black] (-5.51,3.00) circle (0.02);
\draw [fill=black] (-4.00,3.00) circle (0.02);
\draw [fill=black] (-4.00,1.40) circle (0.02);

\draw [fill=black] (0.1988,1.40) circle (0.02);
\draw [fill=black] (0.1988,3.00) circle (0.02);
\draw [fill=black] (1.6891,3.00) circle (0.02);
\draw [fill=black] (1.6821,1.4) circle (0.02);

\node at (-6.45,2.2) {\tiny{$C_-$}};
\node at (2.7,2.2) {\tiny{$C_+$}};
\end{tikzpicture}
\caption{Setup of $\Sigma_l$}
\label{fig2}
\end{figure}

\noindent For $\ge > 0$ let $F^{l,\gd,\ge}$ denote a smooth subharmonic function
on $\Sigma^{l}$ such that
\begin{align*}
 \lim_{\ge \to 0} F^{l,\gd,\ge} = f^{l,\gd} =: F^{l,\gd,0}
\end{align*}
uniformly in ${C^1(\Sigma^{l})}$.  This can be achieved by extending $f^{l,\gd}$
continuously from the boundary and using mollifiers.  Lastly, set
\begin{align*}
 W^{l,m,\gd,\ge}(t) := \int_{-\infty}^{\infty} \gk^m(s) F^{l,\gd,\ge}(s,t)
ds.
\end{align*}
Certainly by construction $W^{l,m,\gd,\ge}$ is a $C^{2,\ga}$ function of $t$,
and so
we compute
\begin{align*}
 \frac{d^2 W^{l,m,\gd,\ge}}{dt^2} =&\ \int_{-\infty}^{\infty} \gk^{m}(s)
\frac{\del^2
F^{l,\gd,\ge}}{\del t^2} (s,t) ds\\
=&\ \int_{-\infty}^{\infty} \gk^{m}(s) \left[ \gD_{s,t} F^{l,\gd,\ge} -
\frac{\del^2
F^{l,\gd,\ge}}{\del s^2} \right](s,t) ds\\
\geq&\ - \int_{-\infty}^{\infty} \gk^{m}(s) \left[ \frac{\del^2
F^{l,\gd,\ge}}{\del
s^2}(s,t) \right] ds\\
=&\ - \int_{-\infty}^{\infty} \frac{d^2 \gk^m}{ds^2} F^{l,\gd,\ge}(s,t) ds\\
=&\ - \frac{1}{\bar{\gk} m^2} \int_{-\infty}^{\infty} \frac{d^2
\gk}{ds^2}\left(\frac{s}{m} \right) F^{l,\gd,\ge}(s,t) ds.
\end{align*}
As noted above $\brs{f^{l,\gd}(s,t)}$ has a bound independent of $l,\gd,s,t$,
and
thus $F^{l,\gd,\ge}$ has a uniform bound independent of $l,\gd,\ge,s,t$.  Thus
\begin{align*}
\frac{d^2 W^{l,m\gd,\ge}}{dt^2} \geq&\ - \frac{C}{\bar{\gk}m^2}
\int_{-\infty}^{\infty} \brs{\frac{d^2 \gk}{ds^2}} \left(\frac{s}{m} \right)
ds\\
=&\ - \frac{C}{\bar{\gk} m} \int_{-\infty}^{\infty} \brs{\frac{d^2 \gk}{d
s^2}}(s) ds\\
\geq&\ - \frac{C}{m}.
\end{align*}
With this estimate on the second derivative and the fact that the boundary
values of $W$ are $o(\gd) + o(\ge)$, by elementary calculus arguments we
conclude that,
for any $t \in [0,1]$,
\begin{align*}
 W^{l,m,\gd,\ge}(t) \leq&\ \frac{C}{m} + o(\gd) + o(\ge).
\end{align*}
Due to the uniform convergence of $F^{l,\gd,\ge}$ to $F^{l,\gd,0}$ as $\ge \to
0$ we may send $\ge$ to $0$ in this estimate to yield
\begin{align} \label{convpf50}
 W^{l,m,\gd,0}(t) \leq \frac{C}{m} + o(\gd).
\end{align}
Since the $\phi^{l,\gd}$ satisfy a uniform $C^{1,1}$ bound, and moreover
solutions to the boundary value problem are unique, it follows that
\begin{align*}
 \sup_{\Sigma^{l} \times M} \nm{\phi^{l} - \phi^{l,\gd}}{C^{1,\ga}} =
o(\gd).
\end{align*}
In particular this convergence implies convergence in the distance topology. 
Furthermore, by (\cite{ChenSOKM3} Proposition 4.6), $\phi^{l}$ converges to
$\phi_t$, the $C^{1,1}$ geodesic connecting $\phi_0$ to $\phi_1$,
in $C^{1,\ga}$ on any fixed compact subset of $\Sigma^{\infty} \times M$.  In
particular, this convergence holds on $\Sigma^{m} \times M$ for any given $m$. 
Thus, given arbitrary $\gg > 0$ we may choose $l$ sufficiently large and $\gd$
sufficiently small so that for all $s,t \in \Sigma^m$, $d(\phi^{l,\gd}(s,t),
\phi_t) \leq \gg$.  Therefore by Lemma \ref{C11lsc} we obtain for
these
choices and all $s,t \in \Sigma^m$,
\begin{align*}
 f^{l,\gd}(s,t) =&\ \nu^{l,\gd}(s,t) - (1-t) \nu(\phi_0) - t \nu(\phi_1)\\
\geq&\ \bar{\nu}(\phi_t) - o(\gg) - (1-t) \nu(\phi_0) - t \nu(\phi_1).
\end{align*}
By the definition of $W$ this implies
\begin{align*}
 W^{l,m,\gd,0} \geq&\ \bar{\nu}(\phi_t) - o(\gg) - (1-t) \nu(\phi_0) - t
\nu(\phi_1).
\end{align*}
Combining this with (\ref{convpf50}) yields
\begin{align*}
 \bar{\nu}(\phi_t) \leq&\ (1-t) \nu(\phi_0) + t \nu(\phi_1) + o(\gg) + o(\gd)
+
\frac{C}{m}.
\end{align*}
Due to the uniform convergence discussed above on $\Sigma^{m}$, we first send
$l$ to infinity and $\gd$ to $0$ (which implies $\gg \to 0$ as discussed above)
to yield
\begin{align*}
 \bar{\nu}(\phi_t) \leq&\ (1-t) \nu(\phi_0) + t \nu(\phi_1) + \frac{C}{m} =
(1-t) \bar{\nu}(\phi_0) + t \bar{\nu}(\phi_1) + \frac{C}{m}.
\end{align*}
Taking the limit as $m \to \infty$ yields the result.
\end{proof}
\end{prop}

\subsubsection{Main Proof}

\begin{proof}[Proof of Theorem \ref{fundexistthm}] Fix $(M^{2n}, \gw, J)$ a
compact K\"ahler manifold.  By Lemmas \ref{triangleineq} and \ref{NPCcompl} it
follows that
$(\bar{\HH}, \bar{d})$ is an
NPC space.  Furthermore, by Lemmas \ref{lscextlsc} and \ref{Kenerglsc} the
function $\bar{\nu} : \bar{\HH} \to \mathbb R$ is lower semicontinuous.  By
Proposition \ref{convexityprop} $\bar{\nu}$ is geodesically convex.
Fix $\phi \in \HH$.  By Theorem \ref{Kenergydecay} we conclude that
\begin{align*}
\liminf_{\psi \in \HH, d(\phi,\psi) \to \infty}
\frac{\nu(\psi)}{d(\phi,\psi)^2} \geq&\ \liminf_{\psi \in \HH,
d(\phi,\psi) \to \infty} \frac{\nu(\phi) - d(\phi,\psi)
\sqrt{\CC(\phi)}}{d(\phi,\psi)^2} = 0.
\end{align*}
By definition this inequality passes to $\bar{\nu}$, and therefore in the
notation of Theorem \ref{Mayer} we have shown $A = 0$.  The theorem follows from
Theorem \ref{Mayer}.
\end{proof}

\subsection{Further properties of minimizing movements} \label{furtherprops}

The theory of minimizing movements comes with a host of a priori regularity
results which seek to exhibit the manner in which these can be thought of as
gradient flows.  We record some of these results here as immediate corollaries
of results in \cite{Mayer}.  First, one can further characterize the paths of
Theorem \ref{Mayer} as curves of
``steepest descent.''  

\begin{defn} Given $(X, d)$ a complete NPC space and $f : X \to \mathbb R$ a
lower semicontinuous function, for $x \in X$ let
\begin{align*}
 \brs{\N_- f}(x) =&\ \max \left\{ \limsup_{y \to x} \frac{f(x) - f(y)}{d(x,y)},
0 \right\}.
\end{align*}
\end{defn}

\begin{thm} (\cite{Mayer} Theorem 2.14) Given the setup of Theorem \ref{Mayer},
if $x(t_0)$ is not a stationary point for $f$, then
 \begin{align*}
  \lim_{t \to t_0^+} \frac{f(x(t_0)) - f(x(t))}{d(x(t), x(t_0))} = \brs{\N_-
f}(x(t_0)).
 \end{align*}
Moreover, for $t_0 > 0$ this limit is finite.
\end{thm}

\begin{thm} Let $(M^{2n}, \gw, J)$ be a compact K\"ahler manifold.  Let $u_t :
[0, \infty) \to \bar{\HH}$ be a K-energy minimizing movement as in Theorem
\ref{fundexistthm}.  Then
\begin{align*}
\lim_{s \to 0^+} \frac{d(u_{t+s}, u_t)}{s} = \brs{\N_- \bar{\nu}}(u_t).
\end{align*}
\begin{proof} This is an immediate corollary of (\cite{Mayer} Theorem 2.17) 
\end{proof}
\end{thm}

\begin{thm} Let $(M^{2n}, \gw, J)$ be a compact K\"ahler manifold.  Let $u_t :
[0, \infty) \to \bar{\HH}$ be a K-energy minimizing movement as in Theorem
\ref{fundexistthm}.  Then for almost all $t > 0$ one has
\begin{align*}
\frac{d \bar{\nu}(u_t)}{dt} = -  \brs{\N_- \bar{\nu}}^2(u_t).
\end{align*}
\begin{proof} This is an immediate corollary of (\cite{Mayer} Corollary 2.18) 
\end{proof}
\end{thm}

The construction of minimizing movements also allows one to derive interesting
properties of the flow map on $\bar{\HH}$.

\begin{defn} Let $(M^{2n}, \gw, J)$ be a compact K\"ahler manifold.  The
\emph{K-energy flow map} is
\begin{align*}
F :&\ \bar{\HH} \times [0, \infty) \to \bar{\HH}\\
:&\ (u_0, t) \to u_t,
\end{align*}
where $u_t$ is the minimizing movement with initial condition $u_0$
guaranteed by Theorem \ref{fundexistthm}.
\end{defn}

\begin{thm} \label{Fprops} (\cite{Mayer} Theorem 2.2, Corollary 2.3, Theorem
2.5, Corollary
2.6)
\begin{enumerate}
\item{Given $u_0 \in \bar{\HH}$ with $\bar{\nu}(u_0) < \infty$, the map $t \to
F_t(u_0)$ is uniformly H\"older continuous of exponent $\frac{1}{2}$.  More
specifically, there exists a constant $C$ such that for any $0 \leq s \leq t$,
one has
\begin{align*}
d(F_s(u_0), F_t(u_0)) \leq C (t-s)^{\frac{1}{2}}.
\end{align*}}
\item{The map $F : \{ \bar{\phi} \in \bar{\HH} | \bar{\nu}(\bar{\phi}) < \infty
\}
\times [0, \infty) \to \bar{\HH}$ is continuous.}
\item{$F$ satisfies the semigroup property, i.e. for $s,t \geq 0$ one has
$F_{s+t} = F_s \circ F_t$.}
\item{Given $u_0 \in \bar{\HH}$ with $\bar{\nu}(u_0) < \infty$, the map $t \to
\bar{\nu}(F_t(u_0))$ is nonincreasing.}
\end{enumerate}
\end{thm}

It was shown by Calabi-Chen that the distance between two points in $\HH$
decreases when each is flowed along Calabi flow.  We reproduce this property
for minimizing movements by adapting an argument from \cite{Mayer}.

\begin{proof}[Proof of Theorem \ref{flowcontr1}] The proof is adapted from
(\cite{Mayer} Lemma 1.12) We will show that the resolvent operator is distance
nonincreasing, and this implies the theorem due to the nature of the convergence
in Theorem \ref{fundexistthm}.  Fix $\psi_0, \psi_1 \in \HH$, $\tau > 0$, and
let $\phi_0 = W_{\tau}(\psi_0), \phi_1 = W_{\tau}(\psi_1)$.  Let $\phi_t$ denote
the geodesic connecting $\phi_0$ to $\phi_1$.  By
using the quadrilateral comparison inequality for NPC spaces (Theorem
\ref{quadcomp}), one has
\begin{align*}
 \frac{1}{2 \tau} d^2(\phi_t, \psi_0) + \frac{1}{2 \tau} d^2(\phi_{1-t}, \psi_1)
\leq&\ \frac{1}{2 \tau} \left[ d^2(\phi_0, \psi_0) + d^2(\phi_1,\psi_1) + 2 t^2
d^2(\phi_0, \phi_1) \right.\\
&\ \left. + t \left( d^2(\psi_0,\psi_1) - d^2(\phi_0,\phi_1) \right) - t \left(
d(\psi_0,\psi_1) - d(\phi_0,\phi_1) \right)^2 \right]
\end{align*}
Note also that the convexity of $\bar{\nu}$ implies that
\begin{align*}
 \bar{\nu}(\phi_t) + \bar{\nu}(\phi_{1-t}) \leq&\ \bar{\nu}(\phi_0) +
\bar{\nu}(\phi_1).
\end{align*}
Combining these two inequalities yields
\begin{align*}
 \FF_{\psi_0,\tau}(\phi_t) + \FF_{\psi_1,\tau}(\phi_{1-t}) \leq&\
\FF_{\psi_0,\tau}(\phi_0) + \FF_{\psi_1,\tau}(\phi_1)\\
&\ - \frac{t}{2\tau} \left(d^2(\phi_0,\phi_1) - d^2(\psi_0,\psi_1) + (
d(\phi_0,\phi_1) - d(\psi_0,\psi_1))^2 \right)\\
&\ + \frac{t^2}{\tau} d^2(\phi_0,\phi_1)\\
\leq&\ \FF_{\psi_0,\tau}(\phi_t) + \FF_{\psi_1,\tau}(\phi_{1-t})\\
&\ - \frac{t}{2\tau} \left(d^2(\phi_0,\phi_1) - d^2(\psi_0,\psi_1) + (
d(\phi_0,\phi_1) - d(\psi_0,\psi_1))^2 \right)\\
&\ + \frac{t^2}{\tau} d^2(\phi_0,\phi_1)\\
\end{align*}
where the second inequality follows from the definition of the resolvent
operator.  It follows that
\begin{align*}
 0 \leq&\ - d^2(\phi_0,\phi_1) + d^2(\psi_0,\psi_1) - (d(\phi_0,\phi_1) -
d(\psi_0,\psi_1))^2 + 2 t d^2(\phi_0,\phi_1).
\end{align*}
Rearranging and sending $t$ to zero gives the result.
\end{proof}

Lastly, one can characterize some convergence properties of minimizing
movements, which bear some relationship to Conjecture \ref{convconj}.

\begin{thm} \label{convergeprop} (\cite{Mayer} Proposition 2.40) Let $(M^{2n},
\gw, J)$ be a compact K\"ahler manifold.
 \begin{enumerate}
  \item {If there exists a cscK metric in $[\gw]$, then $d(\phi_t, \phi_0)$
remains bounded for every $\phi_0 \in \bar{\HH}$, $t > 0$.}
  \item {If there exists $\phi_0 \in \bar{\HH}$ and a sequence $t_n \to \infty$
so that $d(\phi_{t_n}, \phi_0) \leq C$, then there exists a minimizer for
$\bar{\nu}$.}
 \item {If there exists $\phi_0 \in \bar{\HH}$ and a sequence $t_n \to \infty$
so that $\{\phi_{t_n}\} \to \phi_{\infty}$, then $\phi_{\infty}$ is a minimizer
for $\bar{\nu}$.}
\end{enumerate}
\end{thm}

\begin{rmk} It is possible (\cite{Mayer} Theorem 2.42) to guarantee a priori
convergence to a fixed point if one assumes \emph{uniform convexity} for
$\bar{\nu}$.  That is, there exists $\ge > 0$ so that for $\phi : [0,1] \to
\bar{\HH}$ a geodesic,
\begin{align*}
\bar{\nu}(\phi_t) \leq (1-t) \bar{\nu}(\phi_0) + t \bar{\nu}(\phi_1) - \ge
t(1-t) \bar{d}(\phi_0,\phi_1).
\end{align*}
By examining the second variation formula for $K$-energy, one sees that such an
inequality could potentially be shown given some a priori control over the
lowest eigenvalue of the Lichnerowicz Laplacian.  As this operator has a kernel
if and only if the manifold admits holomorphic vector fields, one sees here how
the presence or lack of such vector fields influences the convergence of Calabi
flow.
\end{rmk}

\section{Higher Regularity of Minimizing movements} \label{highregsec}

In this section we prove Theorem \ref{highreg}.  First we
derive a priori estimates for K\"ahler potentials in geodesic balls in the
intersection of geodesic balls of $\HH$ and sublevel sets of $\nu$.  We use
these to control large time steps of discrete solutions to Calabi flow, and
then pass these estimates to the limiting minimizing movement.  

\begin{lemma} \label{Jlowerbnd} Let $(M^{2n}, \gw, J)$ be a compact K\"ahler
manifold satisfying $c_1 < 0$.  There exists a constant $C$ such that for all
$\phi \in \HH$,
\begin{align*}
J(\phi) \geq - C \sqrt{V} d(0,\phi).
\end{align*}
\begin{proof} Fix $\psi \in \HH$, and let $\gg :[0,1] \to \HH$ be the unique
$C^{1,1}$ geodesic connecting $0$ to $\psi$.  We note that, if $C$ denotes a
lower bound for the Ricci curvature of $\gw$,
\begin{align*}
\left. \frac{d}{dt} J \right|_{t=0} =&\ - \int_M \left.\frac{\del \gg}{\del
t}\right|_{t=0} \rho(\gw) \wedge \gw^{n-1}\\
\geq&\ - C \int_M \brs{\left. \frac{\del \gg}{\del t} \right|_{t=0}} \gw^{n}\\
\geq&\ - C \nm{\left.\frac{\del \gg}{\del t} \right|_{t=0}}{L^2(\gw)}
\sqrt{V}\\
=&\ - C E(0)^{\frac{1}{2}} \sqrt{V}\\
=&\ - C d(0,\phi) \sqrt{V},
\end{align*}
where the last equality follows since the energy element along a geodesic is
constant.  By (\cite{ChenLBM} Proposition 2), since $c_1 < 0$, $J$ is convex
along
$C^{1,1}$ geodesics, and hence
\begin{align*}
\frac{d}{dt} J \geq - C \sqrt{V} d(0,\phi)
\end{align*}
for all $t \in [0,1]$.  Integrating this inequality over $[0,1]$ yields the
proposition.
\end{proof}
\end{lemma}

\begin{lemma} \label{suplemma} There exists a constant $C$ so that for all $\phi
\in \HH$,
\begin{align} \label{supbound}
\sup_M \phi \leq&\ \frac{1}{V} I^A(\phi) + \frac{1}{\sqrt{V}} d(0,\phi) + C.
\end{align}
\begin{proof} Since $\tr_{\gw} \left( \gw + \sqrt{-1} \del \delb \phi \right) =
n + \gD_g \phi > 0$, we can integrate against the Greens function for $g$ to
yield
\begin{align*}
\phi(x) =&\ \frac{1}{V} \int_M \phi(y) \gw^n(y) - \frac{1}{V} \int_M \gD \phi 
G(x,y) \gw^n(y)\\
\leq&\ \frac{1}{V} \int_M \phi \gw^n + \frac{n}{V} \int_M G(x,y) \gw^n(y)\\
\leq&\ \frac{1}{V} \int_M \phi \gw^n + C\\
=&\ \frac{1}{V} \left( I^A(\phi) + \int_M \phi \gw_{\phi}^n \right) + C\\
\leq&\ \frac{1}{V} \left( I^A(\phi) + \int_{\phi > 0} \phi \gw_{\phi}^n \right)
+ C\\
\leq&\ \frac{1}{V} I^A(\phi) + \frac{1}{\sqrt{V}} d(0,\phi) + C,
\end{align*}
where the last line follows from Lemma \ref{distanceestimate}.
\end{proof}
\end{lemma}

\begin{lemma} \label{L2distbnd} Given $\phi \in \HH$ such that $I(\phi) = 0$ one
has
\begin{align*}
\nm{\phi}{L^2(\gw)}^2 \leq C(1 + d(0,\phi)) + I^A(\phi).
\end{align*}
\begin{proof} Let $\gg : [0,1] \to \HH$ denote the unique $C^{1,1}$ geodesic
connecting $0$ to $\phi$.  Since $\ddot{\gg} \geq 0$, arguing as in
(\ref{deloc20}) we have that
\begin{align*}
\phi \geq \dot{\gg}(0).
\end{align*}
It follows that if $\phi_- = - \inf \{0, \phi\}$, we have $\phi_-^2 \leq
(\dot{\gg}(0))^2$.  Since $\sup \phi \leq C + \frac{d(0,\phi)}{\sqrt{V}} +
\frac{1}{V} I^A(\phi)$ by
Lemma
\ref{suplemma}, we have that
\begin{align*}
\nm{\phi}{L^2(\gw)}^2 =&\ \int_M \left( \phi_+^2 + \phi_-^2 \right) \gw^n\\
\leq&\ \int_M \left( \left( C + \frac{d(0,\phi)}{\sqrt{V}} + \frac{1}{V}
I^A(\phi) \right) +
\brs{\dot{\gg}(0)}^2 \right) \gw^n\\
\leq&\ C V + \sqrt{V} d(0,\phi) + E(\gg(0)) + I^A(\phi)\\
\leq&\ C + C d(0,\phi) + I^A(\phi).
\end{align*}
\end{proof}
\end{lemma}

\noindent Next we recall two lemmas from the work of Tian.
\begin{lemma} \label{loglb1} (\cite{T1} Proposition 2.1) Let $(M^{2n}, \gw,
J)$ be a compact K\"ahler manifold.  There
exist $\ga, C > 0$ such that for all $\phi \in \HH$ satisfying $\sup_M \phi = 0$
one has
\begin{align*}
\int_M e^{-\ga \phi} \gw^n \leq C.
\end{align*}
\end{lemma}

\begin{lemma} \label{loglb} Let $(M^{2n}, \gw, J)$ be a compact K\"ahler
manifold.  There
exist $\ga, C > 0$ such that for all $\phi \in \HH$ one has
\begin{align*}
\frac{1}{V} \int_M \log \left( \frac{\gw_{\phi}^n}{\gw^n} \right) \gw_{\phi}^n
\geq \frac{\ga}{V} I^A(\phi) - C.
\end{align*}
\begin{proof} The proof is contained in (\cite{T2} pg. 95).  From Lemma
\ref{loglb1} there exists $C, \ga$ such that
\begin{align*}
\frac{1}{V} \int_M \exp \left[ - \log \frac{\gw_{\phi}^n}{\gw^n} - \ga \left(
\phi - \sup \phi \right) \right] \gw_{\phi}^n \leq C.
\end{align*}
Since the exponential function is convex we can apply Jensen's inequality to
yield
\begin{align*}
\frac{1}{V} \int_M \log \left( \frac{\gw_{\phi}^n}{\gw^n} \right) \gw_{\phi}^n
\geq&\ - \frac{\ga}{V} \int_M \left( \phi - \sup \phi \right) \gw_{\phi}^n -
\log
C\\
\geq&\ - \frac{\ga}{V} \left( \int_M \phi \gw_{\phi}^n + V \sup \phi \right) -
C\\
=&\ \frac{\ga}{V} \left( I^A(\phi) - \int_M \phi \gw^n + V \sup \phi \right) -
C\\
\geq&\ \frac{\ga}{V}\left( I^A(\phi) - \sup \phi \int_M \gw^n + V \sup \phi
\right) - C\\
=&\ \frac{\ga}{V} I^A(\phi) - C.
\end{align*}
\end{proof}
\end{lemma}

\begin{prop} \label{mainenergyprop} Let $(M^{2n}, \gw, J)$ be a compact K\"ahler
manifold satisfying $c_1 < 0$.  Given $A$, $B > 0$ there exists $C(A,B) > 0$
such that if $\phi \in
\HH$ satisfies
\begin{itemize}
\item{$I(\phi) = 0$}
\item{$- A \leq \nu(\phi) \leq A$}
\item{$d(0, \phi) \leq B$,}
\end{itemize}
then
\begin{enumerate}
\item{$-C \leq J(\phi) \leq C,$}
\item{$I^A(\phi) \leq C,$}
\item{$J^A(\phi) \leq C,$}
\item{$- C \leq \int_M \log \frac{\gw_{\phi}^n}{\gw^n} \gw_{\phi}^n \leq C.$}
\end{enumerate}
\begin{proof}  By Lemma \ref{Jlowerbnd}, we obtain the lower bound
\begin{align*}
J(\phi) \geq - C B^2.
\end{align*}
On the other hand, using the representation of $K$-energy in (\ref{Kenergform})
and Lemma \ref{loglb} we have
\begin{align*}
J(\phi) =&\ \nu(\phi) - \int_M \log \frac{\gw_{\phi}^n}{\gw^n}
\gw_{\phi}^n\\
\leq&\ A - \ga I^A(\phi) + C\\
\leq&\ A + C.
\end{align*}
Since $J(\phi)$ is thus bounded above and below, turning again to
(\ref{Kenergform}) yields an upper and lower bound on $\int_M \log
\frac{\gw_{\phi}^n}{\gw^n} \gw_{\phi}^n$.  Applying Lemma \ref{loglb} again
yields the upper
bound for $I^A(\phi)$.
\end{proof}
\end{prop}

\begin{proof}[Proof of Theorem \ref{highreg}] Fix $\phi_0 \in \HH$ and let $\phi
:[0,\infty) \to \bar{\HH}$ be the K-energy minimizing movement with initial
condition $\phi_0$.   As guaranteed by (\ref{solnprop1}), we know that
\begin{align*}
\phi_t = \lim_{n \to \infty} W^n_{\frac{t}{n}}(\phi_0).
\end{align*}
It follows from Lemma \ref{distcntrl}
that $d(\phi_0, \phi_t)$ is uniformly controlled in terms of $\nu(\phi_0)$ and
$t$.  Now choose a sequence $\{\phi_t^j\} \in \HH$ converging to $\phi_t$ in the
distance topology.  Certainly $d(\phi_0, \phi_t^j)$ is then uniformly controlled
in terms of $\nu(\phi_0)$ and $t$, and thus $d(0, \phi_t^j)$ is controlled in
terms of $d(0,\phi_0), \nu(\phi_0)$ and $t$ by the triangle inequality. 
Combining this with Lemma \ref{suplemma}, Lemma \ref{L2distbnd} and Proposition
\ref{mainenergyprop} there exists a constant $C = C(t,d(0,\phi_0),\nu(\phi_0))$
such that, for all $j$,
\begin{align*}
\nm{\phi_t^j}{H_1^2} + \sup_M \phi_t^j \leq C.
\end{align*}
By choosing a subsequence we obtain a sequence satisfying these inequalities and
converging weakly in $H_1^2$ and almost everywhere, and so both of these
inequalities pass to the limit as $j \to \infty$, finishing the theorem.
\end{proof}

\section{Smooth Convergence of discrete Calabi flows with uniform bounds}
\label{fundconvsec}

In this section we prove Theorem \ref{fundconv}, which says that a sequence of
discrete Calabi flows with given
initial condition, vanishing step size, and uniform $C^{4,\ga}$ bounds in space
contains a subsequence which converges to a smooth solution of Calabi flow.

\begin{thm} \label{convthm} Suppose that given $T > 0$ there exists $C, \ga >
0$ and a sequence $\{\phi_j^i\}$ of discrete Calabi flows on $[0,
T]$ with uniform step size $\tau_i \to 0$ and initial condition $\phi_0$,
such that
\begin{align} \label{convspat}
\sup_{i,j} \brs{\phi^{i}_j}_{C^{4,\ga}} + \brs{\log
\frac{\gw_{\phi^i_j}^n}{\gw^n}}_{C^{2,\ga}} \leq&\ C,\\
\sup_{i,j} \brs{\phi_{j+1}^i - \phi_j^i - \frac{\del \gg^i_j}{\del t}}_{C^0}
\leq&\
\tau_i o(\tau_i) \label{convcoh}.
\end{align} 
Then there exists a subsequence of $\{\phi^{\ge}_t \}$ converging in $C^{4,
\ga'}, \ga' < \ga,$ to a smooth solution of Calabi flow on $[0, T]$.
\begin{rmk} The hypothesis (\ref{convcoh}) is reasonable to make, as it says
that when the geodesic distance between two very close point is rescaled to unit
length, the curve is approaching a the straight line path between the points,
which should follow from uniqueness of solutions to the geodesic equation.
\end{rmk}

\begin{proof}  It follows directly from Arzela-Ascoli that at any time $t \in
[0, T]$ one obtains a subsequence converging in $C^{4,\ga}$.  This yields a
one-parameter family $\phi_t \in \mathcal H$.  We claim that this family is
differentiable in $t$, and moreover satisfies the Calabi flow.  Fix some time
$t_0 \in [0, T]$, and fix $h > 0$.  We claim there exists constants $C, \gg > 0$
such that
\begin{align} \label{convpropestimate}
P(t_0,h) := \lim_{i \to \infty} \brs{\frac{\phi^i(t_0+ h) - \phi^i(t_0)}{h} -
\left(s_{\phi^i(t_0)} - \bar{s} \right)} = o(h).
\end{align}
Since we have convergence of $\phi^i_t$ to $\phi_t$ in $C^{4,\ga}$, the
theorem will follow from this claim.  As each $\phi^i$ is a discrete Calabi
flow with uniform step size, we have that $\phi^i(t) =
\phi^i_{\floor{\frac{t}{\tau_i}}}$.  Recall also the variational equation
\begin{align} \label{convproof10}
\frac{1}{\tau_i} \left. \frac{\del \gg_j^i}{\del t}
\right|_{t = 1} = s_{{\phi^i_{j+1}}} - \bar{s}.
\end{align}
where $\gg_j^i : [0, 1] \to \mathcal H$ is the unique $C^{1,1}$ geodesic
connecting $\phi^i_j$ and $\phi^i_{j+1}$.  Hence we obtain
\begin{align*}
P(t_0, h) =&\ \lim_{i \to \infty} \brs{\frac{\phi^i_{\floor{\frac{t_0 +
h}{\tau_i}}} - \phi^i_{\floor{\frac{t_0}{\tau_i}}}}{h} - \left(
s_{{\phi^i_{\floor{\frac{t_0}{\tau_i}}}}} - \bar{s} \right)}\\
=&\ \lim_{i \to \infty} \frac{1}{h} \brs{ \left[\sum_{j =
\floor{\frac{t_0}{\tau_i}}}^{\floor{\frac{t_0 + h}{\tau_i}} - 1}\phi^i_{j+1} -
\phi^i_j \right] - h \left( s_{{\phi^i_{\floor{\frac{t_0}{\tau_i}}}}}- \bar{s}
\right)}\\
=&\ \lim_{i \to \infty} \frac{1}{h}\brs{ \left[ \sum_{j =
\floor{\frac{t_0}{\tau_i}}}^{\floor{\frac{t_0 + h}{\tau_i}} - 1}
\frac{\del \gg^i_j}{\del t} \right] - h
\left( s_{{\phi^i_{\floor{\frac{t_0}{\tau_i}}}}}- \bar{s} \right) + E}\\
\leq&\ \lim_{i \to \infty} \frac{1}{h} \brs{ \left[ \sum_{j =
\floor{\frac{t_0}{\tau_i}}}^{\floor{\frac{t_0 + h}{\tau_i}} - 1}
 \frac{\del \gg^i_j}{\del t} \right] - h
\left( s_{{\phi^i_{\floor{\frac{t_0}{\tau_i}}}}}- \bar{s} \right)} + \lim_{i \to
\infty} \frac{1}{h} \brs{E}\\
=:&\ A + B.
\end{align*}
where $E$ is defined by the equality in the line in which it appears, and $A$
and $B$ are the two terms appearing in the penultimate line.  First we estimate
term $B$ using (\ref{convcoh}).

\begin{gather} \label{convproof20}
\begin{split}
B =&\ \lim_{i \to \infty} \frac{1}{h}\brs{ \left[ \sum_{j =
\floor{\frac{t_0}{\tau_i}}}^{\floor{\frac{t_0 + h}{\tau_i}} - 1} \phi^i_{j+1} -
\phi^i_j - \frac{\del \gg^i_j}{\del t}
\right]}\\
\leq&\ \lim_{i \to \infty} \frac{1}{h} \sum_{j =
\floor{\frac{t_0}{\tau_i}}}^{\floor{\frac{t_0 + h}{\tau_i}} - 1} \brs{
\phi^i_{j+1} - \phi^i_j - \frac{\del
\gg^i_j}{\del t} }\\
\leq&\ \lim_{i \to \infty} \frac{C}{h} \sum_{j =
\floor{\frac{t_0}{\tau_i}}}^{\floor{\frac{t_0 + h}{\tau_i}} - 1}
\tau_i o(\tau_i)\\
=&\ \lim_{i \to \infty} \frac{C}{h} \left( \frac{h}{\tau_i} \right) \tau_i
o(\tau_i)\\
=&\ 0.
\end{split}
\end{gather}
Turning to $A$ we first prove a lemma.
\begin{lemma} \label{convlemma2} Given $0 < \ga' < \ga < 1, C > 0$ and $\ge > 0$
there exists $\gd > 0$ so that if $\phi_1, \phi_2 \in \HH$ satisfy
\begin{enumerate}
\item{$\brs{\phi_i}_{C^{4,\ga}} + \brs{\log
\frac{\gw_{\phi_i}^n}{\gw^n}}_{C^{2,\ga}} \leq C$,}\\
\item{$d(\phi_1,\phi_2) \leq \gd$},
\end{enumerate}
then $\brs{\phi_1 - \phi_2}_{C^{4,\ga}} \leq \ge$.
\begin{proof} If the statement were false then we can choose a sequence
$\{\gd_i\}, \gd_i \to 0$, and sequences of functions $\{\phi_1^i, \phi_2^i\}$
satisfying the hypotheses but $\brs{\phi_1^i - \phi_2^i}_{C^{4,\ga}} > \ge$.  By
property (1) we may apply Arzela-Ascoli to obtain a subsequence of
$\{\phi_1^i\}$ converging in $C^{4,\ga'}, \ga' < \ga$ to $\phi_1^{\infty}$, and
likewise one has $\phi_2^{\infty}$.  By property (2) and the estimates of the
Calabi-Yau theorem \cite{CYThm}, one obtains a uniform lower bound on
$\gw_{\phi_j^i}$, and so $\phi_1^{\infty}, \phi_2^{\infty} \in \HH$.  Since
$d(\phi_1^i, \phi_2^i) \leq \gd_i \to 0$, given the uniform bounds on the
metrics $\{\phi_j^i\}$ we conclude from Lemma \ref{distanceestimate} that
\begin{align*}
\lim_{i \to \infty} \nm{\phi_1^i - \phi_2^i}{L^1(\gw)} = 0.
\end{align*}
It follows that $\phi_1^{\infty} = \phi_2^{\infty}$ and so for sufficiently
large $i$ one has $\nm{\phi_1^i - \phi_2^i}{C^{4,\ga'}} \leq \ge$, a
contradiction.
\end{proof}
\end{lemma}
Now note that since the summand defining $A$ consists of
$\frac{h}{\tau_i}$ terms, we can re-express
\begin{align*}
A = \lim_{i \to \infty} \frac{1}{h} \brs{ \sum_{j =
\floor{\frac{t_0}{\tau_i}}}^{\floor{\frac{t_0 + h}{\tau_i}} - 1} \left[
\frac{\del \gg^i_j}{\del t} - \tau_i
\left( s_{\phi^i_{\floor{\frac{t_0}{\tau_i}}}}- \bar{s} \right) \right]}
\end{align*}
Then, inserting the variational equation (\ref{convproof10}) and applying the
triangle inequality we have
\begin{align*}
A =&\ \lim_{i \to \infty} \frac{1}{h} \brs{ \sum_{j =
\floor{\frac{t_0}{\tau_i}}}^{\floor{\frac{t_0 + h}{\tau_i}} - 1} \tau_i \left(
s_{\phi_{j+1}^i} - s_{\phi^i_{\floor{\frac{t_0}{\tau_i}}}} \right)}
\leq \lim_{i \to \infty} \frac{\tau_i}{h} \sum_{j =
\floor{\frac{t_0}{\tau_i}}}^{\floor{\frac{t_0 + h}{\tau_i}} - 1} \brs{
s_{\phi_{j+1}^i} - s_{\phi^i_{\floor{\frac{t_0}{\tau_i}}}}}.
\end{align*}
Next we observe that, by Lemma \ref{distcntrl}, for all $j$ one has
$\phi^i_{j+1} \in B_{C \tau}(\gw_{\phi^i_j})$.  Thus for every $j \in \left[
\floor{\frac{t_0}{\tau_i}}, \floor{\frac{t_0 + h}{\tau_i} - 1} \right]$ there is
a piecewise geodesic curve connecting $\phi^i_j$ to
$\phi^i_{\floor{\frac{t_0}{\tau_i}}}$ consisting of at most $\frac{h}{\tau_i}$
segments each of length no greater than $\tau_i$.  It follows by the triangle
inequality that for all such $j$,
\begin{align*}
d(\phi_j^i, \phi^i_{\floor{\frac{t_0}{\tau_i}}}) \leq C h.
\end{align*}
Again using that the summand describing $A$ consists of $\frac{h}{\tau_i}$
terms, it follows from Lemma \ref{convlemma2} that
\begin{align*}
A \leq&\ \lim_{i \to \infty} \frac{\tau_i}{h} \sum_{j =
\floor{\frac{t_0}{\tau_i}}}^{\floor{\frac{t_0 + h}{\tau_i}} - 1} \ge(C h) =
\ge(Ch).
\end{align*}
This completes the proof of (\ref{convpropestimate}), finishing the theorem.
\end{proof}
\end{thm}

\section{Conclusion}

As is clear from the proof of Theorem \ref{highreg}, more is proved in the sense
that minimizing sequences for each Moreau-Yosida functional satisfy all of the
estimates of Proposition \ref{mainenergyprop}.  Obtaining further regularity
results in this direction is an essential step in overcoming the gap between
Theorem
\ref{fundexistthm} and Conjecture \ref{existconj}.  Many ingenious arguments are
exploited in Chen-Tian's proof of uniqueness of cscK metrics, which ultimately
is a regularity proof, showing that the geodesic connecting two critical points
of $\nu$ is itself smooth.  The metrics in play in this proof already satisfy an
a priori $C^{1,1}$ bound though, making the problem more tractable.  Ultimately,
obtaining stronger a priori estimates on the intersection of geodesic balls and
sublevel sets of $\nu$ will be essential in obtaining higher regularity of
minimizing movements.

\bibliographystyle{hamsplain}

\end{document}